\documentclass[12pt]{amsart}
\usepackage{amsfonts, amssymb, amsmath, mdframed, tikz-cd, comment, anyfontsize}
\usepackage[backend=bibtex,giveninits=true,maxbibnames=99,maxcitenames=99,maxalphanames=99,
sorting=nyt,style=alphabetic,useprefix=true,url=false,doi=false,isbn=false]{biblatex}
\renewbibmacro{in:}{}
\usetikzlibrary{cd,decorations.pathmorphing,decorations.pathreplacing,patterns}
\def\R{\mathbb{R}}

\def\Z{\mathbb {Z}}
\def\F{\mathbb{F}}
\def\ilim{\varprojlim}
\def\dlim{\varinjlim}

\usepackage[colorlinks=true, pdfstartview=FitV, linkcolor=blue, 
      citecolor=blue]{hyperref}
\usepackage{amsfonts, amssymb, amsmath, mdframed, graphicx}
\usepackage{bbm, a4wide}
\usepackage{appendix}

\DeclareMathAlphabet{\mymathbb}{U}{bbold}{m}{n}

\newcommand{\AAA}{\mathcal{A}}

\newcommand{\be}{\begin{equation}}
\newcommand{\ee}{\end{equation}}

\newtheorem{theorem}{Theorem}[section]
\newtheorem{prop}[theorem]{Proposition}
\newtheorem{cor}[theorem]{Corollary}
\newtheorem{lem}[theorem]{Lemma}

\theoremstyle{definition}
\newtheorem{definition}[theorem]{Definition}
\newtheorem{example}[theorem]{Example}
\newtheorem{remark}[theorem]{Remark}

\title{Covers of Tiling Spaces}

\author{Franz G\"ahler}
\address{Fakult\"at f\"ur Mathematik, Universit\"at Bielefeld, \newline
  \indent  Postfach 100131, 33501 Bielefeld, Germany}
\email{gaehler@math.uni-bielefeld.de}

\author{Jianlong Liu}
\address{Department of Mathematics, University of Texas, \newline
  \indent 2515 Speedway, PMA 8.100 Austin, TX 78712, USA}
\email{jlliu@utexas.edu}

\author{Lorenzo Sadun}
\address{Department of Mathematics, University of Texas, \newline
  \indent 2515 Speedway, PMA 8.100 Austin, TX 78712, USA}
\email{sadun@math.utexas.edu}

\begin{document}
\sloppy

\begin{abstract}
  We study the ways that one tiling space can be a finite regular
  cover of another. We classify all of the finite regular covers of a
  tiling space via its structure as an inverse limit space. If the
  tiling space $\Omega$ can be written as an inverse limit
  $\ilim \Gamma_n$, then the \'etale fundamental group of $\Omega$,
  which is defined via a limit of covers, is isomorphic to the inverse
  limit $\hat \pi_1(\Omega) := \ilim \hat \pi_1(\Gamma_n)$ of the
  profinite completions of the fundamental groups
  $\pi_1(\Gamma_n)$. This isomorphism allows us to construct all
  covers of tiling spaces and to use those covers to distinguish
  spaces that have identical cohomology groups.\end{abstract}

\keywords{Aperiodic tilings, finite regular covers, profinite fundamental group,
  \'etale fundamental group}
\subjclass[2020]{52C23, 57M10, 20E18}

\maketitle

\section{Introduction}\label{sec:Introdution}

\subsection{Background} 

Let $T$ be a tiling of $\R^d$. We consider a space $\Omega$ of tilings (often denoted $\Omega_T$) 
that resemble $T$, in the sense that a tiling $T'$ is in
$\Omega$ if every pattern that appears in $T'$ can be found somewhere in $T$. Under standard assumptions about
$T$, $\Omega$ is a compact metric space. The group $\R^d$ of translations acts on $\Omega$, making $\Omega$
an abstract dynamical system. Since the 1990s, people have used topological and dynamical properties of 
$\Omega$, especially the spectrum of the group action and the \v Cech cohomology groups $\check H^k(\Omega)$, 
to understand geometric and statistical properties of tilings like $T$. 

For tilings generated by a substitution $\sigma$, Anderson and Putnam (\cite{MR1631708}) showed how to express 
$\Omega$ as an inverse limit of a compact branched manifold $\Gamma$ under an expansive self-map 
(also denoted $\sigma$) induced by
the substitution. This implies that the \v Cech cohomology 
$\check H^k(\Omega)$ is the direct limit of $H^k(\Gamma)$ under the action of the pullback map $\sigma^\ast$. 
This construction was generalized by G\"ahler (\cite{gaehler02}), by Sadun (\cite{MR2014868}), and by
Bellissard, Benedetti and Gambaudo (\cite{MR2193205}) to apply to all tilings meeting mild assumptions,
not just substitution tilings. As a result, \v Cech cohomology became a useful and readily computable
invariant of tiling spaces \cite{SadBook}. Not only can it be used to tell 
spaces apart, but many important properties of tilings, such as the possible deformations of the geometry of 
the tiles or the rate at which Birkhoff sums converge to ergodic averages, can be expressed in cohomological 
terms (\cite{MR1971208}, \cite{MR2201938}, \cite{MR2437225}, \cite{MR3851781}).

Cohomology, however, is not a complete invariant. There exist pairs of tiling spaces, such as the 
period doubling and Thue--Morse tilings of $\R$, that have isomorphic cohomologies (with the same order structure)
without being homeomorphic spaces. This is not suprising, since $H^1(\Gamma_n)$ is dual to the abelianization of 
$\pi_1(\Gamma_n)$ and since abelianization loses information. To probe the topology of $\Omega$ more deeply, 
it makes sense to look at non-abelian invariants. This can be done either by constructing something that maps
from $\Gamma_n$ to $\Gamma_{n+1}$ and taking a direct limit, or by constructing something that maps from
$\Gamma_{n+1}$ to $\Gamma_{n}$ and taking an inverse limit. 

\begin{itemize}

\item Using the first strategy, we can pick a fixed group $G$ and look at $\textnormal{Hom}(\pi_1(\Gamma_n), G)$. The 
map $\sigma_{n+1}: \Gamma_{n+1} \to \Gamma_n$ induces a pullback map $\sigma_{n+1}^\ast: \textnormal{Hom}(\pi_1(\Gamma_n), G) \to 
\textnormal{Hom}(\pi_1(\Gamma_{n+1}),G)$. We define the {\em representation variety} with group $G$ to be
\[ V_G(\Omega) := \dlim \textnormal{Hom}(\pi_1(\Gamma_n), G). \]
We can likewise study the quotient $V_G(\Omega)/G$ of $V_G(\Omega)$ by inner automorphisms of $G$ 
or the quotient $V_G(\Omega)/\textnormal{Aut}(G)$ by all automorphism of $G$. 
Erdin showed in \cite{erdin2010patternequivariantrepresentationvariety} that
$V_G/G$ (and therefore $V_G$ and $V_G/\textnormal{Aut}(G)$) is a topological invariant,
independent of how we write $\Omega$ as an inverse limit.
Note that if $G$ is abelian, then $\textnormal{Hom}(\pi_1(\Gamma_n),G)=H^1(\Gamma_n, G)$, 
whose direct limit is $\check H^1(\Omega, G)$, making the representation variety a 
non-abelian generalization of (first) cohomology. 

For any element of $V_G(\Omega)$ defined by a map from $\pi_1(\Gamma_n)$ to $G$, and for all $m \ge n$,
the image of $\pi_1(\Gamma_{m+1})$ is a subgroup of the 
image of $\pi_1(\Gamma_m)$. If $G$ is finite, the range stabilizes to 
a subgroup $G_0 \subset G$ and is already accounted for in $V_{G_0}(\Omega)$. For 
information specific to $G$, it makes sense to look at 
$\dlim \textnormal{Sur}(\pi_1(\Gamma_n), G)$ (and its quotients by $G$ or $\textnormal{Aut}(G)$),
where $\textnormal{Sur}(\pi_1(\Gamma_n),G)$ is the set of surjective homomorphisms from $\pi_1(\Gamma_n)$
to $G$.

\item For any group $G$, the {\em profinite completion} $\hat G$ of $G$ is 
the inverse limit of all quotients of $G$ by finite-index normal subgroups. For instance, the
profinite completion of the integers is the universal odometer $\hat \Z = \ilim (\Z/n\Z) 
=\prod_p \ilim_n (\Z/p^n\Z)$, where $p$ ranges over the primes. The map 
$\sigma_{n+1}: \Gamma_{n+1} \to \Gamma_n$ induces a map
${{}\hat{\sigma}_{n+1}}_\ast:\hat \pi_1(\Gamma_{n+1}) \to \hat \pi_1(\Gamma_n)$. 
The {\em profinite fundamental group} of $\Omega$, denoted
$\hat \pi_1(\Omega)$, is defined to be the inverse limit over $n$ of $\hat \pi_1(\Gamma_n)$.
The profinite fundamental group is often easy to compute from a particular presentation of $\Omega$ as
an inverse limit space. However, it is harder to see what $\hat \pi_1(\Omega)$ has to do with the
topology of $\Omega$, or even that it is independent of how we express $\Omega$ as an inverse limit space. 

\item 
We can consider finite regular covers 
of $\Omega$, each associated with a group of deck transformations. 
A finite regular cover $\Omega'$ of $\Omega$ may have its own finite regular covering space $\Omega''$, 
which is also a finite regular cover of $\Omega$. The {\em \'etale fundamental group} (\cite{milneLEC})
of $\Omega$, denoted $\pi_1^\textnormal{et}(\Omega)$, is the inverse limit of the groups of deck
transformations. This is manifestly a topological invariant of $\Omega$, but it is not immediately
clear how to compute $\pi_1^\textnormal{et}$. 

\end{itemize}

Each path component of $\Omega$ is contractible, so the naive definition of 
$\pi_1(\Omega)$ yields a trivial group. This does not imply that $\hat \pi_1(\Omega)$ or 
$\pi_1^\textnormal{et}(\Omega)$ are trivial! Since $\Omega$ is neither path-connected
nor locally simply-connected, the usual correspondence between covers of $\Omega$ and subgroups of 
$\pi_1(\Omega)$ does not apply. Instead, we will study $\hat \pi_1(\Omega)$ and 
$\pi_1^\textnormal{et}(\Omega)$ as alternatives to studying $\pi_1(\Omega)$.

\subsection{Results} In this paper we relate the three constructions. 
First we show how covers of tiling spaces are related to representation varieties: 
\begin{theorem}\label{thm:covers} Let $G$ be a finite group and let $\Omega$ be a 
space of repetitive tilings with finite local complexity. 
Regular covers of $\Omega$ whose group of deck transformations is isomorphic to $G$, modulo homeomorphisms that are lifts of the
identity on $\Omega$, are parametrized by 
\[ \dlim \textnormal{Sur}(\pi_1(\Gamma_n),G)/\textnormal{Aut}(G). \]
\end{theorem} 

Note also that homomorphisms $\phi_n: \pi_1(\Gamma_n) \to G$ are closely related to finite-index
normal subgroups, insofar as the kernel of $\phi_n$ is a 
finite-index normal subgroup $N$ of $\pi_1(\Gamma_n)$. If $\phi_n$ is surjective, then 
$G \simeq \pi_1(\Gamma_n)/N$. However, it is possible for two surjective homomorphisms $\pi_1(\Gamma_n) \to G$ to
have the same kernel, if those homomorphisms are related by an automorphism of $G$. 
By looking at the quotient of $\dlim \textnormal{Sur}(\pi_1(\Gamma_n),G)$ by $\textnormal{Aut}(G)$, we are essentially studying 
a limit of $\ker(\phi_n)$ rather than a limit of $\phi_n$ itself. 

The proof of Theorem \ref{thm:covers} uses an explicit construction. Given a sequence $\Phi$ of surjections 
$\phi_n: \pi_1(\Gamma_n) \to G$ with $\phi_{n+1} = \phi_n \circ \sigma_\ast$, we construct a 
cover of $\Omega$. We then show that all (regular and finite) covers of $\Omega$ are equivalent to 
covers of this form. 
Finally, we show that two sets $\Phi = \{\phi_n\}$ and $\Phi' = \{{\phi}'_n\}$ of surjections give 
equivalent covers if and only if $\phi_n' = A \circ \phi_n$ for all sufficiently large $n$ and for some fixed
automorphism $A: G \to G$. 

Next we relate $\hat \pi_1(\Omega)$ and $\pi_1^\textnormal{et}(\Omega)$:
\begin{theorem}\label{thm:same} In all cases, $\hat \pi_1(\Omega)$ and $\pi_1^\textnormal{et}(\Omega)$ are 
isomorphic. If \(\Omega\) is generated by a substitution \(\sigma:\Gamma\rightarrow\Gamma\), both groups 
are isomorphic to \(\hat{\pi}_1^\sigma(\Gamma)\cong(\pi_1^\textnormal{et})^\sigma(\Gamma)\), where 
\(\hat{\pi}_1^\sigma(\Gamma)\) and \((\pi_1^\textnormal{et})^\sigma(\Gamma)\) are the 
\(\sigma\)-equivariant profinite/\'etale fundamental groups of \(\Gamma\).
\end{theorem}
The proof involves taking the inverse limit, over larger and larger groups $G$, of the 
construction of Theorem \ref{thm:covers}. A limit of covers is an element of $\pi_1^\textnormal{et}(\Omega)$, 
while a limit over $G$ of the kernels of $\phi_n$ gives us $\hat \pi_1(\Omega)$. Theorem \ref{thm:same},
together with the manifest topological invariance of $\pi_1^\textnormal{et}$, makes it clear that $\hat \pi_1(\Omega)$
is a topological invariant, independent of how we write $\Omega$ as an inverse limit. 

Theorems \ref{thm:covers} and \ref{thm:same} can both be proved abstractly, 
considering inverse limits spaces in general and not just tiling spaces. Likewise, the topological invariance 
of $\hat \pi_1(\Omega)$ can be proved directly, without using Theorem \ref{thm:same}. 
However, for applications to tiling theory, we believe that a more concrete proof is preferable. 

\begin{theorem}\label{thm:substitution-covers} Let $\Omega$ be a substitution tiling space with substitution
$\sigma$ and let $p: \tilde \Omega \to \Omega$ be a (finite, regular) covering map. Then 
$\tilde \Omega$ is a substitution tiling space with a substitution $\tilde \sigma$ such that 
$p \circ \tilde \sigma = \sigma^n \circ p$ for some finite integer $n$. 
\end{theorem}

That is, covers of substitution tiling spaces can't necessarily be constructed by modifying the original
substitution map, but they can always be constructed by modifying powers of the original substitution, with 
different covers requiring different powers. 

We illustrate these theorems with a case study of the space $\Omega_\textnormal{Fib}$ of 
Fibonacci tilings in one dimension. We show that 
\begin{theorem}\label{thm:Fib} $\pi_1^\textnormal{et}(\Omega_\textnormal{Fib})=\hat \pi_1(\Omega_\textnormal{Fib}) = \hat{\mathbb{F}}_2$,
where $\mathbb{F}_2$ is the free group on two generators. In particular, every finite 2-generator group 
is the group of deck transformations 
for some cover $\tilde \Omega \to \Omega_\textnormal{Fib}$. $\tilde \Omega$ is a substitution
tiling. $\tilde \Omega$ can also be obtained from a cut-and-project construction (i.e., $\tilde \Omega$ has pure point spectrum) if 
and only if the group of deck transformations is abelian. 
\end{theorem}

In Section \ref{sec:review} we review the definitions of tilings and tiling spaces
and see how one tiling space can be a cover of another. We also review the definitions of 
profinite completions and how homomorphisms of groups induce homomorphisms of their 
profinite completions. 
In Section \ref{sec:covers} we show how to construct covers of tiling spaces and 
thereby prove Theorems \ref{thm:covers} and \ref{thm:substitution-covers}. 
In Section \ref{sec:same} we give precise 
definitions for $\hat \pi_1$ and $\pi_1^\textnormal{et}$ and prove Theorem \ref{thm:same}. 
In Section \ref{sec:Fib} we treat the Fibonacci tiling as a case study
and prove Theorem \ref{thm:Fib}. In Section \ref{sec:examples} we consider the \'etale fundamental 
group of other tiling spaces, with calculations for the more involved examples done in Sage 
using \cite{liu26}, and an independent implementation in GAP by the first author. Finally, in an
Appendix we sketch the connection of our constructions to gauge theory.

\section{Foundations}\label{sec:review}

\subsection{Tiles and tilings}

A {\bf tile} is a closed topological ball in $\R^d$, possibly with a label to distinguish the tile
from other tiles of the same shape and size. A {\bf tiling} of $\R^d$ is a collection of tiles,
intersecting only on their boundaries, whose union is all of $\R^d$. The group $\R^d$ acts on tiles by 
translation and on tilings by simultaneously translating all of the tiles. Two tiles are 
{\bf equivalent} if one is the translate of the other. Equivalence classes of tiles are called
{\bf prototiles}. A {\bf patch} is a finite collection of tiles in a tiling. A tiling has 
{\bf finite local complexity}, or FLC, if there are only finitely many equivalence classes of patches
up to any given diameter. It is straightforward to prove that a tiling has FLC if and only if it has
finitely many prototiles and finitely many ways for two adjacent tiles to touch. 

We say that two tilings are close if they agree on a large ball around the origin up to a small translation.
This topology can be expressed in terms of a metric, but the details of the metric are not 
important for this paper. A {\bf tiling space} is a non-empty set of tilings that is invariant under
translation and closed in the tiling topology. The smallest tiling space containing a specific tiling $T$,
often denoted $\Omega_T$, is called the (continuous) {\bf hull} of $T$. $\Omega_T$ consists of all tilings
$T'$ with the property that all patches in $T'$ can be found somewhere in $T$. $T$ is said to be
{\bf repetitive} if $\Omega_T$ is a minimal dynamical system, that is if 
$\Omega_{T'}=\Omega_T$ for every $T' \in \Omega_T$. Equivalently, $T$ is repetitive if for every 
patch $P \subset T$, there is a radius $R(P)$ such that at least one copy of $P$ appears in every ball
of radius $R(P)$. 

A tiling $T_2$ is said to be {\bf locally derivable} from a tiling $T_1$ if there is a factor map
$\phi: \Omega_{T_1} \to \Omega_{T_2}$ and a radius $R$ such that, whenever two tilings $T, T' \in 
\Omega_{T_1}$ agree on a ball of radius $R$ around a point $x \in \R^d$, $\phi(T)$ and $\phi(T')$
agree exactly on a ball of radius 1 around $x$. If in addition $T_1$ is locally derivable from $T_2$,
we say that $T_1$ and $T_2$ are {\bf mutually locally derivable}, or MLD. We then call $\phi$ an
{\bf MLD equivalence}. Note that this is a stronger condition than simply being topologically conjugate.  

All FLC tilings are MLD equivalent to tilings where the prototiles are polygons (or polyhedra) meeting
full-face to full-face. For simplicity, we will assume throughout that our tilings are of this form, but 
everything we say applies equally well to arbitrary FLC tilings. 

Given an FLC tiling $T$, we can build a CW complex $\Gamma$, called the {\bf Anderson--Putnam complex}, 
by starting with one copy of each 
prototile and identifying points on the boundaries where two prototiles meet. There is a natural map
from $\Omega_T$ to $\Gamma$, taking each tiling to the equivalence class of the origin. In this way,
a point in $\Gamma$ describes what a tiling looks like at the origin. 

Given a tiling $T$, we can obtain another tiling $T_1$ that is MLD to $T$ by adding a label to each 
tile $t$ that
describes the pattern of $T$ in a neighborhood around $t$. This is called {\bf collaring}. A point
in $\Gamma(\Omega_{T_1})$ describes a small neighborhood of the origin in $T_1$, or equivalently a 
larger neighborhood of the origin in $T$, and there is a natural ``forgetful'' map from $\Gamma(\Omega_{T_1})$
to $\Gamma(\Omega_T)$. We can collar out to greater and greater distances, obtaining a sequence of 
MLD tilings $T_i$ and corresponding CW complexes $\Gamma_i=\Gamma(\Omega_{T_i})$ with forgetful maps
$\Gamma_{i} \to \Gamma_{i-1}$. The inverse limit of this sequence is homeomorphic to $\Omega_T$. 

There are many versions of this construction. When $T$ comes from a substitution, Anderson and Putnam's
construction \cite{MR1631708} involves all of the $\Gamma_i$'s being the same space (up to scale) and
the forgetful map is essentially the substitution itself. Other constructions, by G\"ahler
\cite{gaehler02}, Barge--Diamond \cite{MR2383524} and others, are more general (and more complicated). In all cases,
a point in $\Gamma_i$ for $i$ sufficiently large determines a tiling in a given neighborhood of the 
origin and 
a sufficiently large patch around the origin of a tiling determines the corresponding point in any 
given $\Gamma_i$. 

Since $\Gamma_i$ is a CW complex, the algebraic topology of $\Gamma_i$ is relatively simple. 
There is a (Galois) correspondence between covering spaces 
of $\Gamma_i$ and subgroups of $\pi_1(\Gamma_i)$. Different cohomology theories (simplicial, cellular, 
singular, \v Cech) all yield the same groups. If $A$ is any abelian group, $\textnormal{Hom}(\pi_1(\Gamma_i),A)=
H^1(\Gamma_i,A)$. 
The subtlety is understanding what happens when we take inverse limits. 

\subsection{Profinite and \'etale fundamental groups}

\begin{definition}
A group \(G\) is \emph{profinite} if it is an inverse limit of finite groups. It is assigned the usual inverse limit topology.

Given a group \(G\) with the discrete topology, we construct its \emph{profinite completion} as
follows. If $M$ and $N$ are normal subgroups of $G$ (denoted $M, N \trianglelefteq G$) of finite index, and if 
$M$ is a subgroup of $N$ (denoted $M\leq N$), then there are natural projections $\pi_{N \trianglelefteq G}: 
G \to G/N$ and $\pi_{M \trianglelefteq G}: G \to G/M$. There is also a projection  $G/M \to G/N$
that we denote $\pi_{N \trianglelefteq G} \pi_{M \trianglelefteq G}^{-1}$. (Since the kernel of 
$\pi_{M \trianglelefteq G}$ (namely $M$ itself) is contained in the kernel of $\pi_{N \trianglelefteq G}$, 
all points in $\pi_{M \trianglelefteq G}^{-1}$ of an element of $G/M$ are mapped to the same point in $G/N$, 
so this is well defined.) We define
\[
\hat{G}=\varprojlim_{\substack{N\trianglelefteq G\\ [G:N]<\infty}}G/N,
\]
That is, an element of $\hat{G}$ is a collection 
$(g_N)$ of elements of $G/N$, where $N$ ranges over all finite-index normal subgroups of $G$, 
subject to the condition
that if $M\leq N$, then $g_N = \pi_{N\trianglelefteq G}\circ\pi_{M\trianglelefteq G}^{-1}(g_M)$. 
\end{definition}

The following proposition is standard. See, for example, \cite{MR2599132}.
\begin{prop}[Universal Property of Profinite Completion]
\label{prop:universal-property-profinite-completion}
Given \(G\) discrete and \(H\) profinite, a homomorphism \(\phi:G\rightarrow H\)
induces a continuous homomorphism \(\hat{\phi}:\hat{G}\rightarrow H\) so that
\[
\begin{tikzcd}
G\arrow[r,"\eta"]\arrow[rd,"\phi",swap]&\hat{G}\arrow[d,"\hat{\phi}"]\\
&H
\end{tikzcd}
\]
commutes, where \(\eta:G\rightarrow\hat{G}\) is the natural homomorphism induced by the collection 
of projections $\pi_{N\trianglelefteq G}$.

In particular, if the codomain is itself the profinite completion of a discrete group \(H\), i.e. 
\(\eta\circ\phi:G\rightarrow H\rightarrow\hat{H}\), then \(\hat{\phi}:\hat{G}\rightarrow\hat{H}\) is induced by \(G/
\phi^{-1}(N)\rightarrow H/N\). That is, for \(g=(g_M)_{\substack{M\trianglelefteq G\\ [G:M]<\infty}}\in\hat{G}\), 
\(\hat{\phi}(g)=h=(h_N)_{\substack{N\trianglelefteq H\\ [H:N]<\infty}}\) with \(h_N=\pi_{N\trianglelefteq 
H}\circ\phi\circ\pi_{\phi^{-1}(N)\trianglelefteq G}^{-1}(g_{\phi^{-1}(N)})\).
\end{prop}

If $X$ is a connected and locally simply-connected topological space, the \emph{profinite fundamental group} 
of $X$ is usually 
defined to be the profinite completion of the ordinary fundamental group. However, this definition does not work 
for tiling spaces or other inverse limit spaces, as tiling spaces are not locally connected, 
much less locally simply-connected. 
Moreover, the ordinary fundamental group of a 
tiling space (built from paths in a single path component) is trivial, so there is nothing to study. Instead, 
we use a definition adapted to inverse limit spaces built from the induced map between profinite 
completions of discrete groups.

\begin{definition}
  For \(\{\Gamma_i\}_{i\in\mathbb{N}}\), each connected and locally simply-connected,
  \(\sigma_{i+1}:\Gamma_{i+1}\rightarrow\Gamma_i\) the \emph{profinite fundamental group}
  of the inverse limit space \(X=\varprojlim_{\sigma_i}\Gamma_i\) is
\[
\hat{\pi}_1(X)=\varprojlim_{{{}\hat{\sigma}_i}_\ast} \hat{\pi}_1(\Gamma_i)
\]
where the maps ${{}\hat{\sigma}_i}_\ast$ are constructed as in Proposition 
\ref{prop:universal-property-profinite-completion}.
\end{definition}

A priori, this definition appears to depend on the description of $X$ as an inverse limit and not just on the topology
of $X$. For now, we will assume that a specific inverse limit structure has been supplied. This ambiguity 
will be resolved in Section \ref{sec:same}, where we prove that different inverse limit structures yield 
isomorphic groups and where we give an explicit description of the elements of $\hat \pi_1(X)$.

The definition of the \'etale fundamental group, in contrast, is very succinct and has no restrictions on the underlying topological space.
\begin{definition}
Given a connected topological space \(X\), its \emph{\'etale fundamental group} is
\[
\pi_1^\textnormal{et}(X)=\varprojlim_{\substack{X'\textnormal{ finite, regular}\\\textnormal{ cover of }X}}\textnormal{Deck}(X'/X),
\]
where Deck$(X'/X)$ is the group of deck transformations of the cover $X' \to X$ and 
where the inverse limit is over induced maps of the form
\begin{align*}
\begin{tikzcd}[ampersand replacement=\&]
X''\arrow[rr]\arrow[rd]\&\&X'\arrow[ld]\\
\&X\&
\end{tikzcd}
&\begin{tikzcd}[ampersand replacement=\&]\arrow[r,squiggly]\&\vphantom{a}\end{tikzcd}
\begin{tikzcd}[ampersand replacement=\&]
\textnormal{Deck}(X''/X)\arrow[d]\\
\textnormal{Deck}(X'/X)
\end{tikzcd}
\end{align*}
with all maps on the left finite, regular covers.
\end{definition}
The identity map $X \to X$ is a (trivial) regular cover, so the inverse limit is nonempty and the group exists. 
We postpone the construction of the induced maps for tiling spaces to Section \ref{sec:same}, after the discussion 
of the existence of nontrivial finite, regular covers.

\begin{remark}
  The usage of the \'etale fundamental group is a continuation of the theme of an \'etale topology of tiling spaces
  from the second author in \cite{MR4653337} and \cite{liutalk24}, which produces a natural map from
  \(d\)- and \((d-1)\)-dimensional \'etale cohomology groups to K-theory by dualizing \'etale opens and partial
  homeomorphisms on them, and is essentially an isomorphism in low dimensions.
\end{remark}

\section{Constructing covers}\label{sec:covers}

We prove Theorem \ref{thm:covers} in stages. We begin with a well-known example, the double 
cover of the 1-dimensional period doubling tiling space by the Thue--Morse tiling
space. By reversing the operation, we show how to construct infinitely many different covers of 
the period doubling space. We then generalize the procedure to arbitrary 1-dimensional tilings, thereby proving
Theorem \ref{thm:covers} for 1-dimensional tilings. In Subsection \ref{subsec:dimension}, 
we go beyond one dimension and prove Theorem \ref{thm:covers} for all tiling spaces. We then examine
the case of substitution tilings and prove Theorem \ref{thm:substitution-covers}. A reader looking
for the shortest possible proofs can skip ahead to Subsection \ref{subsec:dimension}, but we recommend the 
development of Subsections \ref{subsec:TM} and \ref{subsec:1D} for most readers. 

\subsection{The cover of period doubling by Thue--Morse} \label{subsec:TM}

Thue--Morse tilings are based on a 2-letter alphabet $\{0,1\}$ and the substitution
\[ \sigma_\textnormal{TM}(0) = 01, \qquad \sigma_\textnormal{TM}(1)=10.  \]
A bi-infinite word in the letters 0 and 1 is admissible if every finite sub-word is found in 
$\sigma_\textnormal{TM}^n(0)$ or $\sigma_\textnormal{TM}^n(1)$ for some natural number $n$. We can convert words to tilings by imagining two
kinds of tiles, each of length 1, carrying labels 0 and 1. Let $\Omega_\textnormal{TM}$ be the space of such tilings. 

There is an obvious involution of $\Omega_\textnormal{TM}$, acting without fixed points, that swaps the labels on the tiles,
converting 0's to 1's and vice versa. Let $\Omega_\textnormal{PD}$ be the quotient of $\Omega_\textnormal{TM}$ by this involution, making 
$\Omega_\textnormal{TM}$ a double cover of $\Omega_\textnormal{PD}$. We will show that $\Omega_\textnormal{PD}$ is itself a substitution tiling
space, with alphabet $\{a,b\}$ and tiles of unit length, and with the substitution
\[ \sigma_\textnormal{PD}(a)=ab, \qquad \sigma_\textnormal{PD}(b)=aa. \]
$\sigma_\textnormal{PD}$ is called the {\em period doubling} substitution and the elements of $\Omega_\textnormal{PD}$ are said to be
{\em period doubling tilings}. 

We first collar the Thue--Morse tiles, giving a letter in a Thue--Morse
tiling the subscript $a$ if it is different from the succeeding letter and $b$ if it is the same as the succeeding letter. 
For instance, the sequence 
\[ \ldots 0110100110010110\ldots \]
becomes 
\[ \ldots 0_a1_b1_a0_a1_a0_b0_a1_b1_a0_b0_a1_a0_a1_b1_a0_?\ldots, \]
where the subscript on the final 0 depends on the continuation of the sequence. 
Adding these subscripts does not change the space at all, as there is a continuous bijection between bi-infinite 
sequences of 0's and 1's and resulting sequences of $0_a$'s, $0_b$'s, $1_a$'s and $1_b$'s. 
A sequence with the alphabet $\{0_a, 0_b, 1_a, 1_b\}$ is completely determined by the sequence of 
$a$'s and $b$'s, plus a single 0 or 1. The involution $0 \leftrightarrow 1$ preserves the $a$'s and $b$'s
while flipping that last bit of information, so the quotient $\Omega_\textnormal{PD} := \Omega_\textnormal{TM}/(\Z/2\Z)$ is 
the set of tilings corresponding to allowed sequences of $a$'s and $b$'s, ignoring the 0's and 1's altogether. 
Acting on the alphabet $\{0_a, 0_b, 1_a, 1_b\}$, the Thue--Morse substitution is 
\[ \sigma_\textnormal{TM}(0_a) = 0_a 1_b, \qquad \sigma_\textnormal{TM}(0_b) = 0_a 1_a,
  \qquad \sigma_\textnormal{TM}(1_a) = 1_a 0_b, \qquad \sigma_\textnormal{TM}(1_b) = 1_a 0_a.\]
Once we erase the 0's and 1's and concentrate on the $a$'s and $b$'s, this becomes $\sigma_\textnormal{PD}$. 

We now reverse the procedure, starting with the base space and building a cover. 
In order to cover $\Omega_\textnormal{PD}$ with group $G$ (in this case $\Z/2\Z$), we
assign subscripts in $G$ to every letter of our period doubling sequence. We then apply some 
{\em following rules}:
\begin{itemize}
\item After every $a$ tile, the subscript increases by 1 (mod 2). 
\item After every $b$ tile, the subscript stays the same. 
\end{itemize} 
Thanks to the following rule, we can apply a subscript to a single tile however we wish, but then all other
subscripts are forced. That is, there are two subscripted words for every (finite or infinite) 
un-subscripted word. For instance,
the word $abaaabababaaaba$ corresponds to 
\[ a_0 b_1 a_1 a_0 a_1 b_0 a_0 b_1 a_1 b_0 a_0 a_1 a_0 b_1 a_1 
\quad \hbox{ or } \quad a_1 b_0 a_0 a_1 a_0 b_1 a_1 b_0 a_0 b_1 a_1 a_0 a_1 b_0 a_0 \]
The resulting space of tilings is collared Thue--Morse, only with the roles of the letters and 
the subscripts reversed. 
The group $\Z/2\Z$ of deck transformations acts by adding a constant (0 or 1) to each subscript. 
To obtain uncollared Thue--Morse, just erase the $a$'s and $b$'s and keep the 0's and 1's. 

The lesson is that we don't need to start with a tiling space, such as Thue--Morse, with a free 
$G$-action. We can start with a base space and build a cover with $G$-symmetry by considering letters
with subscripts in $G$ and applying appropriate following rules. 

The same technique can be used to construct other covering spaces of $\Omega_\textnormal{PD}$. For instance, to get
a triple cover we could take subscripts in $\Z/3\Z$ and a following rule such as 
``add 1 after every $a$ and 2 after every $b$''.  In fact, we can generate several different inequivalent 
triple covers of $\Omega_\textnormal{PD}$ this way. As we will soon see, variations on this technique will generate
all covers of 1-dimensional tiling spaces. 

\subsection{Covers of 1-dimensional tiling spaces}\label{subsec:1D}

Let $\Omega$ be an arbitrary 1-dimensional, repetitive, FLC tiling space and let $G$ be a finite group.
We will use a 
``subscripts and following rule'' construction to associate a regular cover of $\Omega$ 
with deck group $G$ to each element of 
$\dlim \textnormal{Sur}(\pi_1(\Gamma_n), G)$. We will then show that every covering space of $\Omega$ 
is equivalent to a cover constructed in this way. After modding out by appropriate equivalences, 
we obtain a complete classification of $G$-covers of $\Omega$. That is, we will prove Theorem \ref{thm:covers}
for 1-dimensional tilings. 

We begin by writing $\Omega = \ilim(\Gamma_n)$ as an inverse limit, 
with each $\Gamma_n$ describing a tiling out to a distance that increases 
with $n$. That is, for each radius $r$ there exists an $n$ such that each point in $\Gamma_n$ determines a 
tiling out to distance $r$, and for each $n$ there is an $R$ such that all tilings that agree on $[-R,R]$
correspond to the same point in $\Gamma_n$. 
We also consider a simplest approximant $\Gamma_0$ consisting of the disjoint union of one copy of each tile in the
alphabet with all vertices identified. 

Any $\Phi \in \dlim \textnormal{Sur}(\pi_1(\Gamma_n), G)$ can be represented by a homomorphism $\phi_n: \pi_1(\Gamma_n) \to G$
for some finite $n$. By collaring out to sufficient distance and then constructing $\Gamma_0$ for the collared 
space, we obtain an approximant with at least as much information as $\Gamma_n$. We can therefore assume, 
without loss of generality, that $n=0$. That is, $\phi_0$ exists and associates an element of $G$ to each letter
in the alphabet $\AAA$ of $\Omega$, at least after collaring sufficiently many times. 

Our tiling space $\tilde \Omega$ will have the alphabet 
$\tilde \AAA := \AAA\times G$, namely letters in $\AAA$ with subscripts in $G$. The allowed sequences of letters in 
$\AAA$ are exactly the same as in $\Omega$, and the corresponding subscripts are required to 
satisfy a following rule:
\begin{itemize}
\item If a letter $\ell$ has subscript $g$, then the following letter must have subscript $g \, \phi_0(\ell)$.
That is, $\phi_0(\ell)$ acts on the right to convert the subscript of $\ell$ into the subscript of the 
next letter. 
\item Our following rule can also be generalized to apply to words and not just letters. 
If $v = \ell_1\ldots\ell_k$ is a word for $\Omega$, 
we say that the subscript for a version of $v$ in $\tilde \Omega$ is the subscript of $\ell_1$. We define 
$\phi_0(v)$ to be $\phi_0(\ell_1)\cdots \phi_0(\ell_k)$. If a version of the concatenated word $vw$ appears in 
$\tilde \Omega$, then the subscript of $w$ must be the subscript of $v$ times $\phi_0(v)$.  We say that $v$ is 
the {\em prefix} of $w$ in the combined word $vw$, and indeed in any word that starts with $vw$, while $w$ is
the {\em suffix} of $v$ in $vw$ or any word that ends in $vw$. The subscript of a suffix is the subscript of 
the prefix times $\phi_0$ of the prefix. 
\end{itemize}
As with our covers of $\Omega_\textnormal{PD}$, an entire tiling in $\tilde \Omega$ is determined by the
corresponding tiling in $\Omega$ and by a single subscript in $G$.  
This makes $\tilde \Omega$ a $G$-cover of $\Omega$, with the group $G$ acting by left multiplication on all of 
the subscripts. 

What remains is showing that $\tilde \Omega$ is a \emph{connected} $G$-cover of $\Omega$. This 
involves showing that each tiling in $\tilde \Omega$ contains all 
versions $P_0 \times G$ of each patch $P_0$ that appears in tilings in $\Omega$. 

\begin{lem}\label{lem:1Dexists} 
For each $g \in G$ there exists a return word $w_g$ of $P_0$ such that 
$\phi_0(w_g)=g$. 
\end{lem}

\begin{proof} 
Let $|P_0|$ be the length of $P_0$ and consider an approximant $\Gamma_{m_0}$ such that each point in 
$\Gamma_{m_0}$ determines a patch of length at least $|P_0|$ near the origin,
with some point in $\Gamma_{m_0}$ determining $P_0$ itself.
Since $\phi_{n_0}$ is surjective, there exist return words for $P_0$ whose images under 
$\phi_{n_0}$ generate 
$G$. Note that if we have a patch containing a collection of copies of $P_0$, and if the return words 
between copies of $P_0$ generate $G$, then the return words starting at the leftmost copy of $P_0$ also
generate $G$. 

Pick a generating set of return words. By the repetitivity of $\Omega$, any sufficiently long
word contains all of these return words. Let $P_1$, of length $|P_1|$, be such a 
long word that begins with $P_0$ and consider 
an approximant $\Gamma_{n_1}$ whose points determine patches of size at least $|P_1|$.
Since $\Phi_{n_1}$ is surjective, there are return words of $P_1$ that generate $G$. 
Pick a word $P_2$ that begins with $P_1$ and is big enough to contain a generating set of these. 
Repeat to recursively define an infinite sequence of $P_k$'s and corresponding $\Gamma_{n_k}$'s.

For each $P_i$, consider the image of all of the return words of $P_{i-1}$ that start at the beginning 
of $P_i$, a subset of $G$ that generates $G$. 
Since $G$ is a finite group, there are only finitely many such subsets. 
By the pigeonhole principle,
one of these subsets must repeat infinitely often. By telescoping, we can then assume that all of the $P_i$'s have 
return words whose images under $\Phi$ are exactly the same. That is, there is a fixed set $\{s_1, \ldots, s_m\}$
of generators of $G$ and return words $w_{k,i}$ of $P_k$ such that $\phi_k(w_{k,i})=s_i$. This also implies that
$\phi_0(w_{k,i})=s_i$. 

Now let $g=s_{i_1} s_{i_2} \dots s_{i_m}$ be an arbitrary element of $G$. 
The word $P_m$ begins with $P_{m-1}$, which begins with $P_{m-2}$, etc., so $P_m$ begins with $P_0$.
By assumption, $w_{m-1,i_1}$ is a prefix for $P_{m-1}$ in $P_m$, $w_{m-2, i_2}$ is a prefix for $P_{m-2}$ in 
$P_{m-1}$, and so on. Let $w_g=w_{m-1,i_1}w_{m-2,i_2}\cdots w_{0,i_m}$. $w_g$ is a prefix for $P_0$ in $P_m$. 
But the empty
word is also a prefix for $P_0$ in $P_m$, so $w_g$ is our desired return word for $P_0$.

\end{proof}

Thus the cover constructed from a particular map $\phi_0$ using $\Gamma_0$ and 
sufficiently collared tiles,
or equivalently from a map $\phi_n$ without collaring, yields a connected regular $G$-cover of $\Omega$. 

\begin{lem}\label{lem:1Dsame}
If two maps $\phi_n$ and $\phi'_m$ represent the same element of 
$\dlim \textnormal{Sur}(\pi_1(\Gamma_n),G)$, then the resulting covering spaces $\tilde \Omega$ and 
$\tilde \Omega'$ are MLD, with the MLD equivalence covering the identity map on $\Omega$. 
\end{lem} 

\begin{proof} The map $\phi_n$ induces maps $\phi_{n+1}, \phi_{n+2}$, etc., and likewise for 
$\phi'_m$. If $\phi_n$ and $\phi'_m$ represent the same element of 
$\dlim \textnormal{Sur}(\pi_1(\Gamma_n),G)$, then there exists an $N$ such that $\phi_N=\phi'_N$. 
This means that $\phi'_m(w)=\phi_n(w)$ for each return word $w$ of  
the patch $P_N$ described by the base point of $\Gamma_N$. We construct an MLD equivalence
$\psi: \tilde \Omega \to \tilde \Omega'$ as follows. If $T \in \tilde \Omega$, then 
\begin{itemize}
\item $\psi(T)$ has tiles in the same locations as $T$ with the same letter labels, but potentially with 
different subscripts. 
\item The first letter of each occurrence of $P_N$ in $\psi(T)$ has the same subscript as in $T$. This is 
consistent with the following rule for $\tilde \Omega'$, since $\phi'_N=\phi_N$. 
\item The subscripts for all other letters in $\psi(T)$ are computed using the following 
rules for $\tilde \Omega'$, starting at the beginning of the previous occurrence of $P_N$. 
\item Since the patch $P_N$ appears in $\Omega$ with bounded gaps, this interpolation is a local operation,
making $\psi$ a local derivation. 
\item The inverse map $\psi^{-1}$ is similar, only taking tilings in $\tilde \Omega'$, 
preserving the subscripts of all $P_N$'s and interpolating using the following rules for $\tilde \Omega$.
\end{itemize}
\end{proof}

Thanks to Lemma \ref{lem:1Dsame}, we can speak of the unique covering space associated to an element of 
$\dlim \textnormal{Sur}(\pi_1(\Gamma_n), G)$. We next show that all finite regular covers of $\Omega$ are of this form. 

\begin{lem}\label{lem:1Dsurjective} Every finite regular cover of $\Omega$ is equivalent to a cover generated by 
an element of $\dlim \textnormal{Sur}(\pi_1(\Gamma_n), G)$.
\end{lem} 

\begin{proof} 
Suppose that $\tilde \Omega$ is a 
connected regular cover of $\Omega$ with
a finite group $G$ of deck transformations and with covering map $p: \tilde \Omega \to \Omega$.
We will construct a surjective map $\phi_n: \pi_1(\Gamma_n) \to G$, where $\Gamma_n$ is an approximant of
$\Omega$. We will then show that the covering space corresponding to $\phi_n$ (which we denote 
$\tilde \Omega_\phi$) is equivalent to $\tilde \Omega$. 

By the definition of a regular covering space, every sufficiently small open set $U$ is evenly-covered, meaning that
$p^{-1}(U)$ is the disjoint union of $|G|$ sets, each mapped homeomorphically to $U$ by $p$, 
with the group $G$ permuting these sets transitively via deck transformations. Now consider an approximant
$\Gamma_n$ such that a point in $\Gamma_n$ corresponds to a patch $P$, such that $U$ contains a 
cylinder set $C$ consisting of all tilings in $\Omega$ that have $P$ at the origin, up
to a small translation. $C$ itself is evenly-covered. Pick one of the preimages and label it $C_e$, where
$e$ is the identity in $G$. The image of $C_e$ by the deck transformation indexed by $g \in G$ is then
labeled $C_g$.

Let $\tilde T \in \tilde \Omega$ be a tiling and let $T = p(\tilde T)$. The return words of $P$
in $T$ correspond to loops in $\Gamma_n$, each of which defines an element of $g$ relating the labels of 
the preimages of the two copies of $P$. This defines a homomorphism $\phi_n: \pi_1(\Gamma_n) \to G$.
We construct a tiling $T'$, MLD to $T$, whose vertices correspond to the locations of $P$ in $T$ and whose edges are
labeled by the paths in $\Gamma_n$ from one vertex to the next. That is, $\Gamma_0$ of the tiling space 
$\Omega_{T'}$ defined by $T'$ is $\Gamma_n$ of $\Omega$. We then build the covering space $\tilde \Omega_\Phi$ 
of $\Omega_{T'}$ (and thus of $\Omega$) using subscripts in $G$ and the following rules defined by $\phi_n$. 

To see that $\tilde \Omega$ and $\tilde \Omega_\Phi$ are equivalent, note that a point in either 
covering space determines
a tiling in $\Omega$, corresponding to one of $|G|$ possible tilings in the other cover. We must show that,
in either direction, local information determines 
which tiling it is. 
First consider tilings in $\tilde \Omega$ or $\tilde \Omega_\Phi$ that project to tilings in 
$\Omega$ with $P$ at the origin. In $\tilde \Omega$, these are indexed by the cylinder set $C_g$ that they lie in, 
while in $\tilde \Omega_{\Phi}$ the tiles themselves carry subscripts in $G$. 
By construction, the labels are the same for
$C_e$ and the transitions are the same, so the labels are the same at all preimages. 
Since $\Omega$ is 
repetitive, there are bounded gaps between these occurrences, so ``look at the nearest occurrence of $P$
to figure out the subscript'' is a local, and therefore continuous, operation. 
\end{proof}

\begin{remark} We have not proven that $\tilde \Omega$ and $\tilde \Omega_\Phi$ are MLD, just that
they are equivalent topologically, being related by a homeomorphism that covers the identity map on $\Omega$. 
While the base tiling space $\Omega$ is always locally derivable from the 
covers constructed with our ``subscripts and following rule'' procedure, local derivability is not part of 
the definition of a covering space. 
If we apply an asymptotically negligible shape change
to a cover $\tilde \Omega_\Phi$, we obtain an equivalent covering space $\tilde \Omega$ such that $\Omega$ is 
not locally derivable from $\tilde \Omega$ and such that $\tilde \Omega$ and $\tilde \Omega_\Phi$ are not MLD. 

However, now that we have shown that all covering spaces are equivalent to covers obtained from subscripts and
following rules, we can restrict our attention to covers where the base space {\em is} locally derivable from the 
cover, and indeed where corresponding tiles have the same sizes in both spaces. 
\end{remark}

\begin{lem}\label{lem:1Dauto}  Let $\Phi$ and $\Phi'$ be elements of $\dlim \textnormal{Sur}(\pi_1(\Gamma_n), G)$, represented by maps 
$\phi_n$ and $\phi'_m$. The resulting covering 
spaces of $\Omega$ are equivalent if and only if $\Phi' = A \circ \Phi$ (that is, if $\phi'_N = A \circ \phi_N$
for all sufficiently large $N$) for some automorphism $A: G \to G$. 
\end{lem} 

\begin{proof} One direction is easy. If $\Phi' = A \circ \Phi$, then applying $A$ to all the subscripts of 
tilings in 
$\tilde \Omega_{\Phi}$ gives a homeomorphism to $\tilde \Omega_{\Phi'}$ that covers the identity on $\Omega$. 

For the other direction, suppose that we have a homeomorphism $h: \tilde \Omega_{\Phi} \to \tilde \Omega_{\Phi'}$
that covers the identity. Since $h$ is uniformly continuous, for any sufficiently long word $P$ and any
fixed subscript $g$ the image $h(C_g)$ of a cylinder set $C_g$ must lie in a single cylinder set $C'_{g'}$ of $P$
in $\tilde \Omega_{\Phi'}$. In particular, if $w$ is a return word of $P$ with $\phi_n(w)=e$, then 
$\phi'_n(w)$ must also be $e$. That is, the kernel of $\phi_N$ for $N$ large lies in the kernel of $\phi'_N$. A similar argument shows that the kernel of $\phi'_N$ lies in
the kernel of $\phi_N$. Since $\phi_N$ and $\phi'_N$ are both 
surjections onto $G$ and have the same kernel, they must be related by an automorphism of $G$. 
That is, $\Phi' = A \circ \Phi$. 
\end{proof}

Taken together, Lemmas \ref{lem:1Dexists}, \ref{lem:1Dsame}, \ref{lem:1Dsurjective} and \ref{lem:1Dauto} prove
Theorem \ref{thm:covers} for 1-dimensional tiling spaces. By Lemma \ref{lem:1Dexists}, there is a map from
$\textnormal{Sur}(\pi_1(\Gamma_n),G)$ to $\{$equivalence classes of finite regular covers of $\Omega\}$.
By Lemma \ref{lem:1Dsame}, this extends to a map on $\dlim \textnormal{Sur}(\pi_1(\Gamma_n),G)$. Lemma \ref{lem:1Dsurjective}
shows that this extended map is surjective and Lemma \ref{lem:1Dauto} shows that the inverse image of any equivalence 
class of covers is an $\textnormal{Aut}(G)$-equivalence class of $\Phi$'s. That is, we have a bijection between 
$\dlim \textnormal{Sur}(\pi_1(\Gamma_n),G)/\textnormal{Aut}(G)$ and $\{$equivalence classes of finite regular covers of $\Omega\}$.

\subsection{Higher dimensional tiling spaces}\label{subsec:dimension}

We now prove Theorem \ref{thm:covers} in all dimensions. There are four steps: 
\begin{enumerate}
\item  To each homomorphism 
$\phi_n: \pi_1(\Gamma_n) \to G$ such that $\phi_m$ is surjective for all $m \ge n$, 
associate a regular, connected covering space $\tilde \Omega$ with deck transformation
group $G$.
\item Show that homomorphisms representing the same element of 
$\dlim \textnormal{Sur}(\pi_1(\Gamma_n),G)$ yield equivalent covering spaces. 
\item Show that every equivalence class of finite, regular covers of $\Omega$ has a 
representative of this form. 
\item Show that two elements $\Phi, \Phi' \in \dlim \textnormal{Sur}(\pi_1(\Gamma_n),G)$ yield equivalent covers if 
and only if $\Phi' = A \circ \Phi$ for some $A \in \textnormal{Aut}(G)$. 
\end{enumerate}

Each step is essentially a repeat of the proof of Lemma \ref{lem:1Dexists}, Lemma \ref{lem:1Dsame}, 
Lemma \ref{lem:1Dsurjective} or Lemma \ref{lem:1Dauto}, only without the simplifying assumption of 1-dimensionality,
which gave us a canonical path from any point in a tiling to any other point. The extra work of dealing with
$d$-dimensional geometry is a significant issue in Step 1, but only in Step 1. 

\noindent {\bf Step 1:} We express $\Omega$ as an inverse limit of approximants such that a point
in $\Gamma_n$ determines a tiling in a region containing the origin, with the region growing to cover all
of $\R^d$ as $n \to \infty$. There are forgetful maps $\Gamma_n \to \Gamma_{n-1}$, since if we know a tiling 
in a bigger region we certainly know it in a smaller region. Any homomorphism $\phi_n: \pi_1(\Gamma_n) \to G$
induces maps $\phi_{n+1}$, $\phi_{n+2}$, etc. We imagine that we are given a surjection $\phi_n$ such that 
every $\phi_m$ with $m \ge n$ is also a surjection. 

To construct a cover of $\Omega$ from $\phi_n$, we first identify a patch $P_0$ determined by a point in 
$\Gamma_n$. We convert any tiling $T \in \Omega$ into a Delone multiset $D$ 
whose points are the locations of $P_0$, 
with each point carrying a label with enough information to reconstruct $T$ from $D$ in a local way. 
The dual Voronoi tiling associated with $D$ is then an FLC and repetitive 
tiling $T'$ that is MLD to $T$, such that the vertices of $T'$ are exactly the locations of $P_0$ in $T$. 
Let $\Omega'$ be the tiling space of $T'$. We will construct of cover of $\Omega'$, which is then also a cover
of $\Omega$. 

To get a tiling in $\tilde \Omega_{\phi_n}$, we assign a subscript in $G$ to every vertex in $T'$. Any path
between vertices $v_1$ and $v_2$ corresponds to a loop in $\Gamma_n$ and so an element of $\pi_1(\Gamma_n)$. 
Different paths from $v_1$ to $v_2$ are homotopic, so all define the same element. In place of a following
rule, we have a {\em return rule}:
\begin{itemize}
\item The subscript at $v_2$ is the subscript at $v_1$ times $\phi_n(\gamma)$, where $\gamma$ is the  
element of $\pi_1(\Gamma_n)$ associated with any path from $v_1$ to $v_2$.
\end{itemize}
This rule gives a consistent labeling of all vertices of $T'$, since $\phi_n$ is a group homomorphism and since
any path from $v_1$ to $v_3$ is homotopic to a path from $v_1$ to $v_2$ followed by a path from $v_2$ to $v_3$. 
The group $G$ acts on these subscripted tilings by left-multiplication of every subscript. The projection to 
$\Omega'$ is just erasing the subscripts. 

As with 1-dimensional tilings, the connectedness of $\tilde \Omega_{\phi_n}$ follows from the repetitivity of 
$\Omega$ and the fact that $\phi_m$ is surjective for all $m \ge n$. We consider a patch $P_1$, based at 
a location featuring the original patch $P_0$, that is big enough to contain a generating set of return vectors from
one $P_0$ to another, that is a set of return vectors whose image under $\phi_m$ generates $G$, 
and thus to contain a generating set of return vectors from a single copy of $P_0$
to the other copies. By our assumption about $\phi_m$, there exists a generating set of return vectors 
of $P_1$, so there exists a patch $P_2$ containing this generating set, and so on. 

By the pigeonhole principle, 
the same set $\{s_1, \ldots, s_k\}$ of generators of $G$ appears infinitely often, so by telescoping we 
can assume that the images of the return vectors of each $P_i$ are the same. For each ordered pair $(i,j)$
of natural numbers, where $j \le k$, we can find a vector $v_{i,j}$, pointing from the ``central'' $P_i$ of $P_{i+1}$
to another copy of $P_i$, whose image under $\phi_n$ is $s_j$. 

Now let $g$ be any element of $G$. We can write $g$ as a product $s_{j_1}s_{j_2}\cdots s_{j_\ell}$ of generators 
of $G$. Starting at the central $P_{\ell-1}$ of $P_\ell$ (which is also the location of a $P_0$), we can 
get to another $P_0$ via the vector $\sum_{i=1}^\ell v_{\ell-i, j_i}$. This is a path from a $P_{\ell-1}$
to another $P_{\ell-1}$ in $P_\ell$, to another $P_{\ell-2}$ in that $P_{\ell-1}$, to another $P_{\ell-3}$ in 
that $P_{\ell-2}$, and so on. $\phi_n$ of this concatenation of paths is then 
$s_{j_1}s_{j_2}\cdots s_{j_\ell}=g$. Thus $P_0$ appears with all possible subscripts in a single tiling. 
Likewise, every other patch appears with all possible subscripts. This shows that $\tilde \Omega_{\phi_n}$ is 
a connected $G$-cover of $\Omega'$, and thus of $\Omega$.

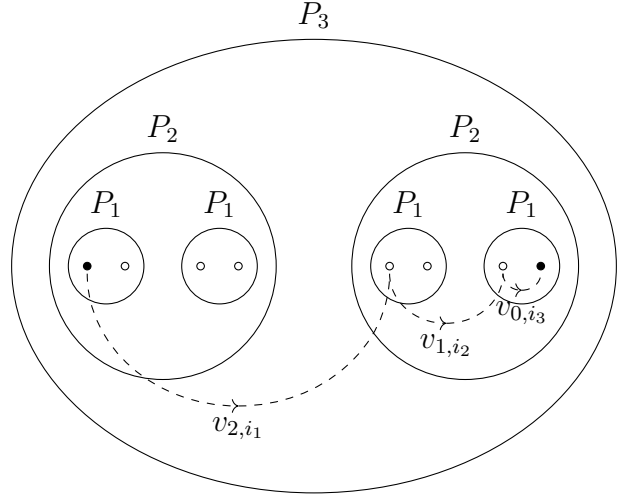
\begin{figure}[t]
\centering
\begin{tikzpicture}
\draw(0,0)circle[x radius=4,y radius=3];
\draw(0,3)node[above]{\(P_3\)};
\draw(-2,0)circle[x radius=1.5,y radius=1.5];
\draw(-2,1.5)node[above]{\(P_2\)};
\draw(2,0)circle[x radius=1.5,y radius=1.5];
\draw(2,1.5)node[above]{\(P_2\)};
\draw(-2.75,0)circle[x radius=0.5,y radius=0.5];
\draw(-2.75,0.5)node[above]{\(P_1\)};
\draw(-1.25,0)circle[x radius=0.5,y radius=0.5];
\draw(-1.25,0.5)node[above]{\(P_1\)};
\draw(1.25,0)circle[x radius=0.5,y radius=0.5];
\draw(1.25,0.5)node[above]{\(P_1\)};
\draw(2.75,0)circle[x radius=0.5,y radius=0.5];
\draw(2.75,0.5)node[above]{\(P_1\)};
\filldraw(-3,0)circle[radius=0.05];
\draw(-2.5,0)circle[radius=0.05];
\draw(-1.5,0)circle[radius=0.05];
\draw(-1,0)circle[radius=0.05];
\draw(1,0)circle[radius=0.05];
\draw(1.5,0)circle[radius=0.05];
\draw(2.5,0)circle[radius=0.05];
\filldraw(3,0)circle[radius=0.05];
\draw[->,dashed](-3,-0.1)arc[x radius=2,y radius=1.75,start angle=-180,end angle=-90];
\draw[dashed](-1,-1.85)arc[x radius=2,y radius=1.75,start angle=-90,end angle=0];
\draw(-1,-1.85)node[below]{\(v_{2,i_1}\)};
\draw[->,dashed](1,-0.1)arc[x radius=0.75,y radius=0.65625,start angle=-180,end angle=-90];
\draw[dashed](1.75,-0.75625)arc[x radius=0.75,y radius=0.65625,start angle=-90,end angle=0];
\draw(1.75,-0.75625)node[below]{\(v_{1,i_2}\)};
\draw[->,dashed](2.5,-0.1)arc[x radius=0.25,y radius=0.21875,start angle=-180,end angle=-90];
\draw[dashed](2.75,-0.31875)arc[x radius=0.25,y radius=0.21875,start angle=-90,end angle=0];
\draw(2.75,-0.31875)node[below]{\(v_{0,i_3}\)};
\end{tikzpicture}
\caption{If $\ell=3$, the vector $\sum_{i=1}^\ell v_{\ell-i, j_i}$ takes us from the base point of $P_3$ to another 
copy of $P_0$.}
\label{fig:Return-vecs}
\end{figure}%

\noindent{\bf Step 2:} Suppose that $\phi_n$ and $\phi'_m$ represent the same element of 
$\dlim \textnormal{Sur}(\pi_1(\Gamma_n),G)$. We must show that $\tilde \Omega_{\phi_n}$ and $\tilde \Omega_{\phi'_m}$ are
equivalent covers. We do this by showing that these spaces are MLD equivalent. 

A tiling $\tilde T \in \tilde \Omega_{\phi_n}$ determines a tiling $T \in \Omega$, which then determines one of 
$|G|$ possible tilings $\tilde T' \in \tilde \Omega_{\phi'_m}$. The only question is whether we can determine the
subscripts in $\tilde T'$ by a local operation from the subscripts in $\tilde T$, and vice versa. 

By assumption, there is an $N$ such that $\phi_N=\phi'_N$. We find a patch $P_N$, corresponding to the base
point of $\Gamma_N$, containing both $P_n$ and $P'_m$. Since the return rules for $P_N$ are the same in 
$\tilde \Omega_{\phi_n}$ and $\tilde \Omega_{\phi'_m}$, we can take the subscripts at each $P_N$ to be the same
in $\tilde T$ and $\tilde T'$ without introducing any inconsistencies. The subscripts at all other points of $\tilde T'$ 
(or $\tilde T$) are determined by the return rules, starting at the nearest $P_N$. Since the $P_N$'s appear with
bounded gaps, this is a local operation. 

\noindent{\bf Step 3:} 

Let $p: \tilde \Omega \to \Omega$ be a 
regular cover of $\Omega$ with
a finite group $G$ of deck transformations and with $\tilde \Omega$ connected.
Consider an approximant
$\Gamma_n$ such that a point in $\Gamma_n$ corresponds to a patch $P$, such that there is evenly-covered  
cylinder set $C$ consisting of all tilings in $\Omega$ that have $P$ at the origin, up
to a small translation. Let $C_e\subset \tilde \Omega$ be one of the $|G|$ preimages of $C$, and for 
each $g \in G$ let $C_g$ be the image of $C_e$ by the deck transformation corresponding to $g$.

Consider a path in a tiling, from one copy of $P$ to another. This corresponds to a loop in $\Gamma_n$. 
The change in subscript between the initial and final instance of $P$ depends only on the path in 
$\Gamma_n$. (If we start with two tilings that agree along a neighborhood of size $|P|$ along the path, 
then the translates in $\Omega$ remain in an evenly-covered neighborhood, so the lifts of these tilings 
must likewise remain close.) Moreover, this change commutes with deck transformations and so corresponds
to right multiplication. The images of actual paths in tilings generate $\pi_1(\Gamma_n)$, so this determines 
a homomorphism $\phi_n: \pi_1(\Gamma_n) \to G$. We then build the cover $\tilde \Omega_{\phi_n}$ as in 
Step 1. 

We claim that $\tilde \Omega$ and $\tilde \Omega_{\phi_n}$ are equivalent. A point in either 
covering space determines a tiling $T \in \Omega$, hence one of $|G|$ possible tilings in the other cover. 
If $T$ has a copy of $P$ at the origin then in $\tilde \Omega$ we are in a specific cylinder set $C_g$ and the
corresponding tiling in $\tilde \Omega_{\phi_n}$ has a vertex at the origin with the same subscript $g$.  
Since $\Omega$ is repetitive, there is always a copy of $P$ in some fixed neighborhood of the origin, 
so we can translate it to the origin, equate the subscripts, and translate back. This operation is local and therefore continuous. 

\noindent{\bf Step 4:} The last step is almost word-for-word the same as in one dimension. 
If $\Phi' = A \circ \Phi$, then applying $A$ to all the subscripts of 
tilings in 
$\tilde \Omega_{\Phi}$ gives a homeomorphism to $\tilde \Omega_{\Phi'}$ that covers the identity on $\Omega$. 

For the other direction, suppose that we have a homeomorphism $h: \tilde \Omega_{\Phi} \to \tilde \Omega_{\Phi'}$
that covers the identity. Since $h$ is uniformly continuous, for any sufficiently large patch $P$ and any
fixed subscript $g$ the image $h(C_g)$ of a cylinder set $C_g$ must lie in a single cylinder set $C'_{g'}$ of $P$
in $\tilde \Omega_{\Phi'}$. In particular, if $w$ is a return vector of $P$ with $\phi_n(w)=e$, then 
$\phi'_n(w)$ must also be $e$. That is, the kernel of $\phi_N$ for $N$ large lies in the kernel of $\phi'_N$.
A similar argument shows that the kernel of $\phi'_N$ lies in
the kernel of $\phi_N$. Since $\phi_N$ and $\phi'_N$ are both 
surjections onto $G$ and have the same kernel, they must be related by an automorphism of $G$. 
That is, $\Phi' = A \circ \Phi$. 

This completes the proof of Theorem \ref{thm:covers}.

\subsection{Substitution structures}

We now prove Theorem \ref{thm:substitution-covers}. If $\Omega$ is a substitution tiling, then 
$\Omega = \ilim (\Gamma, \sigma)$, where all the approximants $\Gamma_i$ equal a single complex $\Gamma$.
Since $\pi_1(\Gamma)$ is finitely generated, for each finite group $G$ the set $\textnormal{Hom}(\pi_1(\Gamma),G)$
has bounded size, so $\dlim \textnormal{Sur}(\pi_1(\Gamma_n),G) \subset \dlim \textnormal{Hom}(\pi_1(\Gamma_n), G)$ is finite. 
By Theorem \ref{thm:covers}, this means that there are only finitely many $G$-covers of 
$\Omega$, up to equivalence.

For each non-negative integer $n$, let $p_n = \sigma^n \circ p$. Since 
$p$ is a $G$-cover of $\Omega$ by $\tilde \Omega$ and $\sigma$ is a single cover of $\Omega$ by itself, 
$p_n$ is a $g$-covering map. Since there are only finitely many covering maps up to equivalence, this 
sequence of covering maps must be eventually periodic up to equivalence. 
That is the same thing as saying the maps $\Phi$ that define
$p$ are eventually periodic up to an automorphism of $G$. 
We can replace the first few terms in the sequence $\Phi$ without changing
the limit, so as to make the entire sequence periodic. But then there is an $n$ such that $p_n$ is 
equivalent to $p$. 

By the definition of equivalence of covers, there exists a map $\tilde \sigma: \tilde \Omega \to 
\tilde \Omega$ such that $p \circ \tilde \sigma = p_n = \sigma^n \circ p$. That is, the 
map $\tilde \sigma$ on 
$\tilde \Omega$ covers the substitution $\sigma^n$ on $\Omega$, making $\tilde \Omega$ a substitution
tiling space. \qed 

We will examine an example of this periodicity of covers, and the construction of $\tilde \sigma$ from $\sigma^n$,
when we study the Fibonacci tiling in Section \ref{sec:Fib}. 

\section{\'Etale vs.~profinite $\pi_1$}\label{sec:same}

\indent In this section, we prove Theorem \ref{thm:same}. We begin by giving a clearer picture of the elements
of \(\hat{\pi}_1(\Omega)\). This picture allows us to then connect to
\(\varinjlim_{{\sigma_i}_\ast}\textnormal{Sur}(\pi_1(\Gamma_i),G)/\textnormal{Aut}(G)\).
We then show that existence of maps between deck transformation groups induced by taking covers of covers,
culminating in a proof of Theorem \ref{thm:same}, and obtaining topological invariance of \(\hat{\pi}_1(\Omega)\)
as a consequence.\\
\indent Recall that the definition of the profinite fundamental group of an inverse limit space is built on the
induced map between profinite completions of discrete groups. That is, given \(G\) and \(H\) discrete and a map
\(\phi:G\rightarrow H\), the map of interest, \(\hat{\phi}:\hat{G}\rightarrow\hat{H}\), is induced by
\(G/\phi^{-1}(N)\rightarrow H/N\). To provide a coherent picture, let us simplify \(\hat{\phi}\) by stating a few
properties of the basic building blocks, maps of the form \(G/\phi^{-1}(N)\rightarrow H/N\).

\begin{prop}[Stabilization of \(G/\phi^{-1}(N)\rightarrow H/N\)]
\label{prop:stabilization}
Let \(G\) and \(H\) be discrete and finitely-generated, \(\phi:G\rightarrow H\), and \(N\trianglelefteq H\).
Then \(G/\phi^{-1}(N)\) is isomorphic to a subgroup of \(H/N\). In fact, the induced homomorphism on quotients
\(\phi_N:G/\phi^{-1}(N)\rightarrow H/N\) is injective.\\
\indent Furthermore, for each \(i\in\mathbb{N}\), let \(\phi_{i+1}:G_{i+1}\rightarrow G_i\).
If \(N=N_0\trianglelefteq G_0\) with \([G_0:N_0]<\infty\), let \(N_i=(\phi_i\circ\cdots\circ\phi_1)^{-1}(N_0)\),
and \(\phi_{N,i+1}:G_{i+1}/N_{i+1}\rightarrow G_i/N_i\) be the induced homomorphisms on quotients.
Then there is an \(n\in\mathbb{N}\) so that for all \(i\in\mathbb{N}\), \(i\geq n\), \(\phi_{N,i}\)
are isomorphisms. In particular, if each \(G_i=G\) and \(\phi_i=\phi\), then there are
\(m,n\in\mathbb{N}\) so that for all \(i\in\mathbb{N}\), \(N_m=N_{m+in}\), and thus \(G/N_m=G/N_{m+in}\).
\begin{proof}
  \(\phi^{-1}(N)=\textnormal{ker}(\pi_{N\trianglelefteq H}\circ\phi)\). By the First Isomorphism Theorem,
  \(G/\phi^{-1}(N)=G/\textnormal{ker}(\pi_{N\trianglelefteq H}\circ\phi)\cong\textnormal{im}(\pi_{N\trianglelefteq H}\circ\phi)\leq H/N\).\\
  \indent For injectivity of \(\phi_N\), given \(g\in G/\phi^{-1}(N)\) so that
  \(\phi_N(g)=\pi_{N\trianglelefteq H}\circ\phi\circ\pi_{\phi^{-1}(N)\trianglelefteq G}^{-1}(g)=1\), \(\phi\circ\pi_{\phi^{-1}(N)\trianglelefteq G}^{-1}(g)\in N\),
  and \(\pi_{\phi^{-1}(N)\trianglelefteq G}^{-1}(g)\in\phi^{-1}(N)\), thus \(g=1\).\\
  \indent Given \(\phi_{i+1}:G_{i+1}\rightarrow G_i\) and \(N_0\trianglelefteq G_0\) with \([G_0:N_0]<\infty\),
  consider the function \(f:\mathbb{N}\rightarrow\mathbb{N}\) with \(i\mapsto |G_i/N_i|\). This is well-defined
  by the previous part, and is a nonincreasing function that is bounded from below. Thus \(f\) has a limit
  that it converges to after finitely many terms, say \(k\in\mathbb{N}\). Since, for all \(i\in\mathbb{N}\),
  \(i\geq k\), \(\phi_{N,i+1}\) is injective on finite groups of the same order, it is surjective, therefore an isomorphism.\\
  \indent Finally, if each \(G_i=G\) and \(\phi_i=\phi\), by Galois Correspondence, for each \(i\in\mathbb{N}\),
  there is some regular cover \(p_i:X_i\rightarrow K(G,1)\) so that \(N_i={p_i}_\ast(\pi_1(X_i))\). Since there are
  only finitely many regular covers of a fixed degree up to equivalence, there are
  \(m,n\in\mathbb{N}\), \(m\geq k\), so that for all \(i\in\mathbb{N}\), \(p_m\) and \(p_{m+in}\) are isomorphic.
  By Galois Correspondence, \(N_m={p_m}_\ast(\pi_1(X_m))={p_{m+in}}_\ast(\pi_1(X_{m+in}))=N_{m+in}\).
\end{proof}
\end{prop}

\noindent Each sequence \(\{N_i\}_{i\in\mathbb{N}}\) satisfying the latter part of Proposition \ref{prop:stabilization}
is associated to a sequence of quotients of the form \(\{G_i/N_i\}_{i\in\mathbb{N}}\), and is key to understanding
\(\varprojlim_{\hat{\phi}_i}\hat{G}_i\). The case where each \(G_i=G\) and \(\phi_i=\phi\) is particularly useful
as it mirrors the case of a substitution tiling space. We give each of these names.

\begin{definition}
  Let \(\{\phi_{i+1}:G_{i+1}\rightarrow G_i\}_{i\in\mathbb{N}}\). Given \(n\in\mathbb{N}\), let \(N_n\trianglelefteq G_n\)
  with \([G_n:N_n]<\infty\). For each \(i\in\mathbb{N}_{\geq n+1}\), denote \(N_i=(\phi_i\circ\cdots\circ\phi_{n+1})^{-1}(N_n)\)
  and \(\phi_{{N_n},i+1}:G_{i+1}/N_{i+1}\rightarrow G_i/N_i\) the induced homomorphism on quotients.
  Each \(G_i/N_i\) is a \emph{vertex}, and \(\phi_{N_n,i+1}\) an \emph{edge}. The collection
  \(\{\phi_{N_n,i}\}_{i\in\mathbb{N}_{\geq n+1}}\) is a \emph{path}, and \(G_n/N_n\) is a \emph{root}.\\
  \indent In this case each \(G_i=G\) and \(\phi_i=\phi\), and there are \(m\in\mathbb{N}_{\geq n}\) and
  \(\ell\in\mathbb{N}\) so that for all \(i\in\mathbb{N}\), \(N_m=N_{m+i\ell}\), the sequence
  \(\{N_i\}_{i\in\mathbb{N}_{\geq n}}\)  (or just \(N_n\)) is \emph{eventually periodic of period \(\ell\)}.
  It is \emph{periodic} if we can choose \(m=0\). Similarly for the corresponding path.
\end{definition}

\noindent In comparison to the setup in Proposition \ref{prop:stabilization}, we have slightly generalized
the definition of paths so roots need not be at level-\(0\).

\begin{remark} When looking for periodicities, it is not enough to require  
\(G/\phi^{-m}(N)\cong G/\phi^{-(m+i)}(N)\). We need the groups $\phi^{-m}(N)$ and 
$\phi^{-(m+i)}(N)$ to be literally the same. If $G=\pi_1(X)$, where $X$ is connected 
and locally simply connected, distinct subgroups $N, N'$ with $G/N \cong G/N'$ correspond to 
different covers of $X$ and to potentially different elements of $\pi_1^{\textnormal{et}}(X)$.
\end{remark}

\indent Using Proposition \ref{prop:stabilization}, we can now describe the inverse limit of induced maps on
profinite completions of a sequence of homomorphisms of discrete groups of the form
\[
\begin{tikzcd}
\cdots\arrow[r]&G_{i+1}\arrow[r,"\phi_{i+1}"]&G_i\arrow[r,"\phi_i"]&\cdots\arrow[r,"\phi_2"]&G_1\arrow[r,"\phi_1"]&G_0.
\end{tikzcd}
\]

\begin{prop}[Elements of \(\varprojlim_{\hat{\phi}_i}\hat{G}_i\)]
\label{prop:elements-projlim}
Let \(\{G_i\}_{i\in\mathbb{N}}\) be groups with the discrete topology, and \(\phi_{i+1}:G_{i+1}\rightarrow G_i\).
\(\varprojlim_{\hat{\phi}_i}\hat{G}_i\) has the structure of a disjoint union of generalized Borel--Bratteli
diagrams whose elements are paths whose roots need not be at level-\(0\) or uniformly fixed,
both among the roots of a single component and among all components.\\
\indent Each component in the disjoint union is a tree and therefore has one tail equivalence class.
It also has a natural partial order with minimal paths. In particular, if each \(G_i=G\) and \(\phi_i=\phi\),
each component has a minimal path with its root at level-\(0\), namely a periodic \(N\).
\begin{proof}
  For each \(i\in\mathbb{N}\), let the \(V_i\), the vertex set at level-\(i\), be the set of all finite quotients
  of \(G_i\), and let \(E_{i+1}\), the edge set from \(V_{i+1}\) to \(V_i\), correspond to
  \(\phi_{i+1,N_i}:G_{i+1}/\phi_i^{-1}(N_i)\rightarrow G_i/N_i\). By construction, there is at most one path
  between any two vertices, thus each component is a tree.\\
  \indent Define a partial order on \(V_i\) where \(G_i/N_i\preceq G_i/M_i\) if \(M_i\leq N_i\),
  so that there is a projection \(\pi_{N_i\trianglelefteq G_i}\circ\pi_{M_i\trianglelefteq G_i}^{-1}:G_i/M_i\rightarrow G_i/N_i\),
  and \(\phi_{i+1}^{-1}(M_i)=\phi_{i+1}^{-1}(N_i)\). From the tree structure, this induces a partial order on the
  set of edges out of any vertex. Since the partial order on the edges out of a vertex is compatible with the
  order of the groups, there is a lower bound on the edges, and there are minimal paths.\\
  \indent If each \(G_i=G\) and \(\phi_i=\phi\), by Proposition \ref{prop:stabilization}, any
  \(N\trianglelefteq G\) is eventually periodic of period \(n\in\mathbb{N}\) starting at some \(m\in\mathbb{N}\).
  Taking \(M=\phi^{-{n\lfloor(m+n-1)/n\rfloor}}(N)\) then gives, for all \(i\in\mathbb{N}\), \(i\geq m\), \(M_i=N_i\),
  and \(M\) is thus periodic. For minimality of \(M\), suppose that there is an \(N\trianglelefteq G\) so that
  \(\phi^{-1}(N)=\phi^{-1}(M)\) and \(M\leq N\). Then
\[
\begin{tikzcd}
G/\phi^{-1}(M)\arrow[r,"\phi_M"]\arrow[rd,swap,"\phi_N"]&G/M\arrow[d,two heads,"\pi_{N\trianglelefteq G}\circ\pi_{M\trianglelefteq G}^{-1}"]\\
&G/N
\end{tikzcd}
\]
commutes with \(\phi_M\) an isomorphism and \(\phi_N\) an injection by Proposition \ref{prop:stabilization}.
Thus \(\phi_N\) is an isomorphism, as is the quotient map \(\pi_{N\trianglelefteq G}\circ\pi_{M\trianglelefteq G}^{-1}\),
which is the identity.
\end{proof}
\end{prop}

\begin{remark}
  Paths in this circumstance are more general than those arising from ordinary ordered (Borel--)Bratteli diagrams.
  In particular, roots need not be at level-\(0\), and minimum paths need not exist. In fact, there may be roots
  at level \(n\) for arbitrarily large \(n\in\mathbb{N}\), and there may be multiple incomparable vertices that
  are groups of the same order with edges out of the same vertex on a prior level!
\end{remark}

\begin{remark}
  Minimal paths are \emph{not} elements of \(\varprojlim_{\hat{\phi}_i}\hat{G}_i\)! They, however, act as
  representatives of each tree. In fact, any path does. In particular, if each \(G_i=G\) and \(\phi_i=\phi\),
  periodic paths act as representatives.\\
  \indent Furthermore, since each \(\phi_{i+1,N_i}\) is injective, each path \(\varprojlim_{\phi_{i+1,N_i}}G_i/N_i\)
  can be written as a direct limit instead (eventual range of the opposite map).
\end{remark}

\indent Finally, applying this machinery with each \(G_i=\pi_1(\Gamma_i)\) and
\(\phi_{i+1}={\sigma_{i+1}}_\ast:\pi_1(\Gamma_{i+1})\rightarrow\pi_1(\Gamma_i)\), we obtain that the elements of the
profinite fundamental group of an inverse limit space are given in Proposition \ref{prop:elements-projlim}. Maps of the form
\[
  \pi_{N\trianglelefteq\pi_1(\Gamma_i)}\circ{\sigma_{i+1}}_\ast\circ\pi_{{\sigma_{i+1}}_\ast^{-1}(N)\trianglelefteq
    \pi_1(\Gamma_{i+1})}^{-1}:\pi_1(\Gamma_{i+1})/{\sigma_{i+1}}_\ast^{-1}(N)\rightarrow\pi_1(\Gamma_i)/N
\]
are important, and we will abbreviate them by \({{\sigma_{i+1}}_\ast}_N\).\\
\indent The following proposition describes the connection between \(\hat{\pi_1}(\Omega)\) and elements
of \(\varinjlim_{{\sigma_i}_\ast}\textnormal{Sur}(\pi_1(\Gamma_i),G)/\textnormal{Aut}(G)\).

\begin{prop}
\label{prop:profinite-deck-transformation-isomorphic}
Elements of \(\varinjlim_{{\sigma_i}_\ast}\textnormal{Sur}(\pi_1(\Gamma_i),G)/\textnormal{Aut}(G)\) correspond to
the tail equivalence classes of \(\hat{\pi}_1(\Omega)\) where the paths have vertices eventually isomorphic to \(G\).
\begin{proof}
This is a direct consequence of the identification
\begin{align*}
\textnormal{Sur}(\pi_1(\Gamma_i),G)/\textnormal{Aut}(G)&\leftrightarrow\left\{\pi_1(\Gamma_i)/N\cong G\right\}\\
[\phi_i]&\leftrightarrow\pi_1(\Gamma_i)/\textnormal{ker}\,\phi_i
\end{align*}
applied to the fact that maps
\[
\textnormal{Sur}(\pi_1(\Gamma_i),G)\rightarrow\textnormal{Sur}(\pi_1(\Gamma_i),G)
\]
are given by
\[
\phi_i\mapsto\phi_i\circ{\sigma_{i+1}}_\ast
\]
restricted to the elements of \(\textnormal{Sur}(\pi_1(\Gamma_i),G)\) that map to surjections. Note that since
both \(\phi_i\) and \(\phi_i\circ{\sigma_{i+1}}_\ast\) are surjections, the First Isomorphism Theorem gives
\[
\pi_1(\Gamma_{i+1})/\textnormal{ker}(\phi_i\circ{\sigma_{i+1}}_\ast)\cong\pi_1(\Gamma_i)/\textnormal{ker}\,\phi_i
\]
via the map \({{\sigma_{i+1}}_\ast}_{\textnormal{ker}\,\phi_i}\).\\
\indent Elements of \(\varinjlim_{{\sigma_i}_\ast}\textnormal{Sur}(\pi_1(\Gamma_i),G)/\textnormal{Aut}(G)\)
are infinite sequences. The identification sends \([\phi_i]\) in the sequence to the quotient
\(\pi_1(\Gamma_i)/\textnormal{ker}\,\phi_i\), and the map \([\phi_i]\mapsto[\phi_i\circ{\sigma_{i+1}}_\ast]\)
is sent to \(\pi_1(\Gamma_i)/\textnormal{ker}\,\phi_i\mapsto\pi_1(\Gamma_{i+1})/\textnormal{ker}(\phi_i\circ{\sigma_{i+1}}_\ast)
=\pi_1(\Gamma_{i+1})/({\sigma_{i+1}}_\ast)^{-1}(\textnormal{ker}\,\phi_i)\), which is exactly the inverse of
\({{\sigma_{i+1}}_\ast}_{\textnormal{ker}\,\phi_i}\).
\end{proof}
\end{prop}

\noindent A few remarks are in order.

\begin{remark}
  An infinite sequence of surjections in \(\varinjlim_{{\sigma_i}_\ast}\textnormal{Sur}(\pi_1(\Gamma_i),G)/\textnormal{Aut}(G)\)
  corresponds to a path in \(\hat{\pi}_1(\Omega)\) (not necessarily rooted at level-\(0\)) where all of the vertices are isomorphic,
  which is what allows us to flip the arrows around and turn the inverse limit of \({\sigma_i}_\ast\) from \(\hat{\pi}_1(\Omega)\)
  into the direct limit of surjections.\\
  \indent Furthermore, the composition \(\phi_i\mapsto\phi_i\circ{\sigma_{i+1}}_\ast\) described in the proof above
  \emph{does not require that \({\sigma_{i+1}}_\ast:\pi_1(\Gamma_{i+1})\rightarrow\pi_1(\Gamma_i)\) be surjective!}
\end{remark}

\begin{remark}
  Noting that, in addition to this proposition, the \'etale fundamental group yields information on finite,
  regular covers of other finite, regular covers, we observe that there is more algebraic structure to
  \(\varinjlim_{{\sigma_i}_\ast}\textnormal{Sur}(\pi_1(\Gamma_i),G)/\textnormal{Aut}(G)\) as well,
  by changing the target group \(G\). In particular, \(\phi:G\twoheadrightarrow H\) induces a map
\[
  \varinjlim_{{\sigma_i}_\ast}\textnormal{Sur}(\pi_1(\Gamma_i),G)/\textnormal{Aut}(G)\rightarrow
  \varinjlim_{{\sigma_i}_\ast}\textnormal{Sur}(\pi_1(\Gamma_i),H)/\textnormal{Aut}(H)
\]
if and only if there is an \(n\in\mathbb{N}\) so that for all \(i\in\mathbb{N}_{\geq n}\),
\[
\begin{tikzcd}
{[\phi_i]}\arrow[r,mapsto]\arrow[d,mapsto]&{[\phi_i\circ{\sigma_{i+1}}_\ast]}\arrow[d,mapsto]\\
{[\phi\circ\phi_i]}\arrow[r,mapsto]&{[\phi\circ\phi_i\circ{\sigma_{i+1}}_\ast]}
\end{tikzcd}
\]
commutes.
\end{remark}

\indent Before translating the last remark into terms more aligned with Galois Correspondence for
ordinary topological spaces, let us summarize key parts of Theorem \ref{thm:covers}.

\begin{prop}
\label{prop:deck-transformation-commute}
Suppose that \(p:\Omega'\rightarrow\Omega\) is a finite, regular cover with
\(\Omega=\varprojlim_{\sigma_i}\Gamma_i\) and \(\Omega'=\varprojlim_{\sigma_i'}\Gamma_i'\)
given by Theorem \ref{thm:covers}, chosen so that, for some \(n\in\mathbb{N}\),
for all \(i\in\mathbb{N}_{\geq n}\), the maps \(p_i:\Gamma_i'\rightarrow\Gamma_i\)
are projections associated to quotients of \(\pi_1(\Gamma_i)\). Let
\begin{align*}
\pi_i&:\Omega\rightarrow\Gamma_i\\
\pi_i'&:\Omega'\rightarrow\Gamma_i'
\end{align*}
be the natural quotient maps. Then, for any pair \(g\in\textnormal{Deck}(\Omega'/\Omega)\)
and \(g_i\in\textnormal{Deck}(\Gamma_i'/\Gamma_i)\) associated by the natural isomorphism
\(\textnormal{Deck}(\Omega'/\Omega)\cong\textnormal{Deck}(\Gamma_i'/\Gamma_i)\) from Proposition
\ref{prop:profinite-deck-transformation-isomorphic}, the diagram
\[
\begin{tikzcd}
\Omega'\arrow[rd,"g"]\arrow[dd,swap,"p"]\arrow[rrrr,"\pi_i'"]&&&&\Gamma_i'\arrow[rd,"g_i"]\\
&\Omega'\arrow[ld,"p"]\arrow[rrrr,"\pi_i'"]&&&&\Gamma_i'\arrow[ld,"p_i"]\\
\Omega\arrow[rrrr,"\pi_i"]&&&&\Gamma_i\arrow[from=uu,crossing over,swap,"p_i" fill=white]
\end{tikzcd}
\]
commutes.
\begin{proof}
  From Theorem \ref{thm:covers}, \(p_i\circ\pi_i'=\pi_i\circ p\) and, from Proposition
  \ref{prop:profinite-deck-transformation-isomorphic}, the isomorphism
  \(\textnormal{Deck}(\Omega'/\Omega)\cong\textnormal{Deck}(\Gamma_i'/\Gamma_i)\) is given by the map
  \(\textnormal{Deck}(\Omega'/\Omega)\rightarrow\textnormal{Deck}(\pi_i'(\Omega')/\pi_i(\Omega_i))
  =\textnormal{Deck}(\Gamma_i'/\Gamma_i)\), thus the top square commutes as well.
  The two sides commute by definition of covers.
\end{proof}
\end{prop}

\indent We have the following corollary as an immediate consequence.

\begin{cor}
\label{cor:deck-transformation-correspondence}
Following the same setup as Proposition \ref{prop:deck-transformation-commute}, for any \(x_i\in\Gamma_i\)
and any \(y_i\in p_i^{-1}(x_i)\), there is a correspondence
\[
\pi_i^{-1}(x_i)\leftrightarrow(\pi_i')^{-1}(y_i)
\]
induced by \(p_i\) and \(p\). In fact, there is a correspondence
\[
\bigsqcup_{\pi_i^{-1}(x_i)}p_i^{-1}(x_i)\leftrightarrow p^{-1}(\pi_i^{-1}(x_i))
\]
induced by \(\pi_i\) and \(\pi_i'\).
\begin{proof}
  Since \(p\) and \(p_i\) are regular covers,
  \(\textnormal{Deck}(\Omega'/\Omega)\ \rotatebox[origin=c]{-90}{\(\circlearrowright\)}\ \Omega'\)
  and \(\textnormal{Deck}(\Gamma_i'/\Gamma_i)\ \rotatebox[origin=c]{-90}{\(\circlearrowright\)}\ \Gamma_i'\)
  are free and transitive on each of the respective fibers. From Proposition \ref{prop:deck-transformation-commute},
  picking any \(x_i\in\Gamma_i\) and any \(y_i\in p_i^{-1}(x_i)\), the stabilizer of \(y_i\) over \(x_i\) is trivial,
  as is the stabilizer of each \(p^{-1}(x)\cap(\pi_i')^{-1}(y_i)\) over any given \(x\in\pi_i^{-1}(x_i)\),
  both due to freeness of the respective action. Transitivity of the corresponding action ensures that
  \(p^{-1}(x)\cap(\pi_i')^{-1}(y_i)\) is a singleton. Adding up among all \(x\in\pi_i^{-1}(x_i)\) gives the
  correspondence \(\pi_i^{-1}(x_i)\leftrightarrow(\pi_i')^{-1}(y_i)\).\\
\indent The second correspondence is obtained by adding up among all \(y_i\in p_i^{-1}(x_i)\), followed by reindexing.
\end{proof}
\end{cor}

\indent We are now ready to describe the connection between deck transformation groups of subcovers for finite,
regular covers of tiling spaces.

\begin{prop}
\label{prop:subcovers}
Suppose that
\[
\begin{tikzcd}
\Omega''\arrow[rr,"r"]\arrow[rd,swap,"q"]&&\Omega'\arrow[ld,"p"]\\
&\Omega
\end{tikzcd}
\]
with \(p\), \(q\), and \(r\) covers, and \(p\) and \(q\) finite and regular, and suppose that
\(\Omega=\varprojlim_{\sigma_i}\Gamma_i\) so that \(\Omega'=\varprojlim_{\sigma_i'}\Gamma_i'\)
and \(\Omega''=\varprojlim_{\sigma_i''}\Gamma_i''\) are given by Theorem \ref{thm:covers},
chosen so that, for some \(n\in\mathbb{N}\), for all \(i\in\mathbb{N}_{\geq n}\),
the maps \(p_i:\Gamma_i'\rightarrow\Gamma_i\) and \(q_i:\Gamma_i''\rightarrow\Gamma_i\)
are projections associated to quotients of \(\pi_1(\Gamma_i)\). Then for all \(i\in\mathbb{N}_{\geq n}\), \(r\)
induces covers \(r_i:\Gamma_i''\rightarrow\Gamma_i'\) so that the squares
\[
\begin{tikzcd}
\Gamma_{i+1}''\arrow[r,"\sigma_{i+1}''"]\arrow[d,"r_{i+1}"]&\Gamma_i''\arrow[d,"r_i"]\\
\Gamma_{i+1}'\arrow[r,"\sigma_{i+1}'"]&\Gamma_i'
\end{tikzcd}
\]
commute. Furthermore, \(r\) is regular.
\begin{proof}
For each \(i\in\mathbb{N}_{\geq n}\), let
\begin{align*}
\pi_i&:\Omega\rightarrow\Gamma_i\\
\pi_i'&:\Omega'\rightarrow\Gamma_i'\\
\pi_i''&:\Omega''\rightarrow\Gamma_i''
\end{align*}
be the natural quotient maps. We want to show that \(r_i\) is induced by the diagram
\[
\begin{tikzcd}
\Omega''\arrow[rd,"r"]\arrow[dd,swap,"q"]\arrow[rrrr,"\pi_i''"]&&&&\Gamma_i''\arrow[rd,dotted,"r_i"]\\
&\Omega'\arrow[ld,"p"]\arrow[rrrr,"\pi_i'"]&&&&\Gamma_i'\arrow[ld,"p_i"]\\
\Omega\arrow[rrrr,"\pi_i"]&&&&\Gamma_i\arrow[from=uu,crossing over,swap,"q_i" fill=white]
\end{tikzcd}.
\]
Let \(x_i''\in\Gamma_i''\). By Corollary \ref{cor:deck-transformation-correspondence},
\((\pi_i'')^{-1}(x_i'')\subseteq\Omega''\) over \(x_i''\in\Gamma_i''\) corresponds to
\(\pi_i^{-1}(q_i(x_i''))\subseteq\Omega\) over \(q_i(x_i'')\in\Gamma_i\), and
\(r((\pi_i'')^{-1}(x_i''))\subseteq\Omega'\) over \(\pi_i^{-1}(q_i(x_i''))\subseteq\Omega\)
corresponds to \(\pi_i'\circ r((\pi_i'')^{-1}(x_i''))\subseteq\Gamma_i'\) over
\(q_i(\{x_i''\})\subseteq\Gamma_i\). Since \(q=p\circ r\), each of which is a quotient map, we have correspondences
\[
(\pi_i'')^{-1}(x_i'')\leftrightarrow\pi_i^{-1}(q_i(x_i''))\leftrightarrow r((\pi_i'')^{-1}(x_i''))
\]
induced by \(q\) and \(r\), respectively, the latter of which gives that \(\pi_i'\circ r((\pi_i'')^{-1}(x_i''))\subseteq\Gamma_i'\)
is a singleton. By the Universal Property of Quotient Topology, \(r_i\) is the unique continuous map so that the top square commutes.\\
\indent Commutativity of the square in the statement of the proposition is a standard diagram chase of
\[
\begin{tikzcd}
&\Omega''\arrow[ld,swap,"\pi_{i+1}''"]\arrow[rrrr,"r"]&&&&\Omega'\arrow[ld,swap,"\pi_{i+1}'"]\arrow[dd,"\pi_i'"]\\
\Gamma_{i+1}''\arrow[rd,swap,"\sigma_{i+1}''"]\arrow[rrrr,"r_{i+1}"]&&&&\Gamma_{i+1}'\arrow[rd,swap,"\sigma_{i+1}'"]&\\
&\Gamma_i''\arrow[rrrr,"r_i"]\arrow[from=uu,crossing over,"\pi_i''" fill=white]&&&&\Gamma_i'
\end{tikzcd}
\]
where the rear top and front squares commute due to the above argument, and the two sides from the Universal Property of Inverse Limits.\\
\indent \(r_i\) a cover, since given any \(x_i'\in\Gamma_i'\) and any sufficiently small neighborhood \(U'\) of \(x_i\), \(U'\) is
homeomorphic to \(p_i(U')\), and \(q_i^{-1}(p_i(U'))\) is a collection of disjoint copies of \(U'\), as is \(r_i^{-1}(U')\).\\
\indent Regularity of \(r_i\) is the standard argument using Galois Correspondence, where
\({q_i}_\ast(\pi_1(\Gamma_i''))\trianglelefteq\pi_1(\Gamma_i)\) and \({q_i}_\ast(\pi_1(\Gamma_i''))
=(p_i\circ r_i)_\ast(\pi_1(\Gamma_i''))\leq{r_i}_\ast(\pi_1(\Gamma_i'))\trianglelefteq\pi_1(\Gamma_i)\),
so \({q_i}_\ast(\pi_1(\Gamma_i''))\trianglelefteq{r_i}_\ast(\pi_1(\Gamma_i'))\), and
\({r_i}_\ast(\Gamma_i'')\trianglelefteq\pi_1(\Gamma_i')\). Regularity of \(r\) then follows
from an application of Theorem \ref{thm:covers}.
\end{proof}
\end{prop}

\begin{cor}
\label{cor:induced-map-deck-transformation}
Following the same setup as Proposition \ref{prop:subcovers}, there is a map
\[
\textnormal{Deck}(\Omega''/\Omega)\rightarrow\textnormal{Deck}(\Omega'/\Omega).
\]
\begin{proof}
Proposition \ref{prop:deck-transformation-commute} and Proposition \ref{prop:subcovers} give that this map is induced by
\[
\textnormal{Deck}(\Gamma_i''/\Gamma_i)\rightarrow\textnormal{Deck}(\Gamma_i'/\Gamma_i)
\]
which exists because \({q_i}_\ast(\pi_1(\Gamma_i''))\leq{p_i}_\ast(\pi_1(\Gamma_i'))\),
and Galois Correspondence and the Universal Property of Quotient Groups give
\[
  \textnormal{Deck}(\Gamma_i''/\Gamma_i)\cong\pi_1(\Gamma_i)/{q_i}_\ast(\pi_1(\Gamma_i''))\rightarrow
  \pi_1(\Gamma_i)/{p_i}_\ast(\pi_1(\Gamma_i'))\cong\textnormal{Deck}(\Gamma_i'/\Gamma_i).
\]
\end{proof}
\end{cor}

\indent While Theorem \ref{thm:covers} constructs finite, regular covers of tiling spaces and expresses them in inverse limits,
it only does so for the paths in the profinite fundamental group where the deck transformation groups of the individual
approximants are isomorphic to the deck transformation of the overall cover. To better illuminate differences between
the \'etale and the profinite fundamental groups, the next proposition does so for each path in the profinite fundamental group.

\begin{prop}[Covers from paths]
\label{prop:construct-covers}
There is a finite, regular cover associated to each path in \(\hat{\pi}_1(\Omega)\).
\begin{proof}
  Let \(\{N_i\}_{i\in\mathbb{N}_{\geq n}}\) be such that, for all \(i\in\mathbb{N}_{\geq n}\),
  \(N_i\trianglelefteq\pi_1(\Gamma_i)\), \(N_{i+1}={\sigma_{i+1}}_\ast^{-1}(N_i)\), and there is no
  \(N\in\pi_1(\Gamma_{n-1})\) so that \(N_n={\sigma_n}_\ast^{-1}(N)\). By Galois Correspondence,
  for each \(i\in\mathbb{N}_{\geq n}\), there is a finite, regular cover \(p_i:\Gamma_i'\rightarrow\Gamma_i\),
  unique up to equivalence, so that \({p_i}_\ast(\pi_1(\Gamma_i'))=N_i\), and whose deck transformation group
  is then \(\pi_1(\Gamma_i)/N_i\). We then have the maps
\[
\begin{tikzcd}
\Gamma_{i+1}'\arrow[d,"p_{i+1}'"]&\Gamma_i'\arrow[d,"p_i'"]\\
\Gamma_{i+1}\arrow[r,"\sigma_{i+1}"]&\Gamma_i
\end{tikzcd}.
\]
By construction, \({\sigma_{i+1}}_\ast\circ {p_{i+1}}_\ast'(\pi_1(\Gamma_{i+1}'))={\sigma_{i+1}}_\ast(N_{i+1})\leq N_i
={p_i}_\ast'(\pi_1(\Gamma_i'))\), thus the Lifting Criterion guarantees that \({\sigma_{i+1}}_\ast\circ {p_{i+1}}_\ast'\)
lifts to some \(\sigma_{i+1}'\) so that
\[
\begin{tikzcd}
\Gamma_{i+1}'\arrow[r,"\sigma_{i+1}'"]\arrow[d,"p_{i+1}'"]&\Gamma_i'\arrow[d,"p_i'"]\\
\Gamma_{i+1}\arrow[r,"\sigma_{i+1}"]&\Gamma_i
\end{tikzcd}.
\]
commutes.
\end{proof}
\end{prop}

\noindent The next three statements give exact criteria on distinguishing finite, regular covers up to isomorphism.

\begin{prop}[Covers from the same path space]
\label{prop:isomorphic-covers}
Any two paths in the same path space of \(\hat{\pi}_1(\Omega)\) yield isomorphic covers.
\begin{proof}
  Let \(\{\sigma_{i+1}':\Gamma_{i+1}'\rightarrow\Gamma_i'\}_{i\in\mathbb{N}_{\geq n}}\) and
  \(\{\sigma_{i+1}'':\Gamma_{i+1}''\rightarrow\Gamma_i''\}_{i\in\mathbb{N}_{\geq n}}\)
  be covers from two paths in the same path space, with maps \(\{p_i'\}_{i\in\mathbb{N}_{\geq n}}\)
  and \(\{p_i''\}_{i\in\mathbb{N}_{\geq n}}\), respectively. The path space is a tree, thus the
  two paths are tail equivalent, say at some \(m\in\mathbb{N}\). Then for all
  \(i\in\mathbb{N}_{\geq m}\), \({p_i}_\ast'(\pi_1(\Gamma_i'))={p_i}_\ast''(\pi_1(\Gamma_i''))\),
  and \(\Gamma_i'\) and \(\Gamma_i''\) are isomorphic. This induces an isomorphism between
  \(\varprojlim_{\sigma_i'}\Gamma_i'\) and \(\varprojlim_{\sigma_i''}\Gamma_i''\).
\end{proof}
\end{prop}

\begin{prop}[Different path spaces correspond nonisomorphic covers]
\label{prop:nonisomorphic-covers}
Any two covers associated to paths in different path spaces of \(\hat{\pi}_1(\Omega)\) are not isomorphic.
\begin{proof}
  Let \(\{\sigma_{i+1}':\Gamma_{i+1}'\rightarrow\Gamma_i'\}_{i\in\mathbb{N}_{\geq n}}\) and
  \(\{\sigma_{i+1}'':\Gamma_{i+1}''\rightarrow\Gamma_i''\}_{i\in\mathbb{N}_{\geq n}}\)
  be covers from two paths in different path spaces, with maps \(\{p_i'\}_{i\in\mathbb{N}_{\geq n}}\)
  and \(\{p_i''\}_{i\in\mathbb{N}_{\geq n}}\), respectively.
\end{proof}
\end{prop}

\noindent Thus the paths of \(\hat{\pi}_1(\Omega)\) associated to elements of
\(\varinjlim_{{\sigma_i}_\ast}\textnormal{Sur}(\pi_1(\Gamma_i),G)/\textnormal{Aut}(G)\)
can be thought of as representatives of equivalence classes of path spaces in \(\hat{\pi}_1(\Omega)\).

\begin{cor}[Criterion for isomorphic covers]
\label{cor:isomorphic-cover-criterion}
If each \(\Gamma_i=\Gamma\) and \(\sigma_i=\sigma\), then finite, regular covers that arise from periodic
\(N\trianglelefteq\pi_1(\Gamma)\) are isomorphic if and only if their corresponding periodic paths share the same root.
\end{cor}

\begin{remark}
  In particular, if each \(\Gamma_i=\Gamma\) and \(\sigma_i=\sigma\), and \(N\trianglelefteq\pi_1(\Gamma)\)
  is periodic of period \(m\in\mathbb{N}_{\geq 2}\), then the covers corresponding to \(N\) and \(\sigma_\ast^{-1}(N)\)
  are not isomorphic.
\end{remark}

\indent Finally, we construct a map \(\hat{\pi}_1(\Omega)\rightarrow\pi_1^\textnormal{et}(\Omega)\),
and obtain an exact criterion for finite deck transformation groups.

\begin{theorem}[Theorem \ref{thm:same}]
\label{thm:profinite-to-etale-isomorphism}
\(\hat{\pi}_1(\Omega)\cong\pi_1^\textnormal{et}(\Omega)\). In particular, if each \(\Gamma_i=\Gamma\) and \(\sigma_i=\sigma\),
then \(\hat{\pi}_1(\Omega)\cong\hat{\pi}_1^\sigma(\Gamma)\cong(\pi_1^\textnormal{et})^\sigma(\Gamma)\cong\pi_1^\textnormal{et}(\Omega)\),
where \(\hat{\pi}_1^\sigma(\Gamma)\) and \((\pi_1^\textnormal{et})^\sigma(\Gamma)\) are the
\(\sigma\)-equivariant profinite/\'etale fundamental groups of \(\Gamma\).
\begin{proof}
We first construct a well-defined map \(\hat{\pi}_1(\Omega)\rightarrow\pi_1^\textnormal{et}(\Omega)\).\\
\indent By Proposition \ref{prop:elements-projlim}, an element in \(\hat{g}\in\hat{\pi}_1(\Omega)\)
is supported on a collection of tail equivalence classes, each represented by some minimal path.
By Proposition \ref{prop:construct-covers}, each path gives a finite, regular cover
\(\Omega'=\varprojlim_{\sigma_i'}\Gamma_i'\) of \(\Omega\), with the associated deck transformation group
given by the inverse limit of quotients of \(\pi_1(\Gamma_i)\) associated to the path, and the projection
of \(\hat{g}\) to this path gives the corresponding deck transformation. Since all paths in the same
tail equivalence class are isomorphic by Proposition \ref{prop:isomorphic-covers}, they are all
associated to the same deck transformation group, and therefore \(\hat{g}\) projected to each path
gives the same deck transformation. The associated deck transformation of the cover is unique by
Proposition \ref{prop:nonisomorphic-covers}, and is independent of the representative in the
isomorphism class of covers of \(\Omega'\) by Theorem \ref{thm:covers}, thus we obtain a well-defined map
\(\hat{\pi}_1(\Omega)\rightarrow\textnormal{Deck}(\Omega'/\Omega)\).\\
\indent Finally, let \(\Omega''\) be another finite, regular cover of \(\Omega\) such that
the diagram
\[
\begin{tikzcd}
\Omega''\arrow[rr]\arrow[rd]&&\Omega'\arrow[ld]\\
&\Omega&
\end{tikzcd}
\]
commutes. By Corollary \ref{cor:induced-map-deck-transformation}, there is a map
\(\textnormal{Deck}(\Omega''/\Omega)\rightarrow\textnormal{Deck}(\Omega'/\Omega)\). By Theorem \ref{thm:covers}
and Proposition \ref{prop:subcovers}, \(\Omega''\) can be written as an inverse limit
\(\varprojlim_{\sigma_i''}\Gamma_i''\) that eventually commutes with projections to the respective approximants
and bonding maps \(\sigma_i\) and \(\sigma_i'\), so that, for sufficiently large \(n\in\mathbb{N}\), for all
\(i\in\mathbb{N}_{\geq n}\), \(\textnormal{Deck}(\Omega'/\Omega)\cong\textnormal{Deck}(\Gamma_i'/\Gamma_i)\)
and \(\textnormal{Deck}(\Omega''/\Omega)\cong\textnormal{Deck}(\Gamma_i''/\Gamma_i)\). By construction,
\[
\begin{tikzcd}
&\textnormal{Deck}(\Gamma_i''/\Gamma_i)\arrow[dd]\\
\hat{\pi}_1(\Omega)\arrow[ru]\arrow[rd]&\\
&\textnormal{Deck}(\Gamma_i'/\Gamma_i)
\end{tikzcd}
\]
commutes, therefore
\[
\begin{tikzcd}
&\textnormal{Deck}(\Omega''/\Omega)\arrow[dd]\\
\hat{\pi}_1(\Omega)\arrow[ru]\arrow[rd]&\\
&\textnormal{Deck}(\Omega'/\Omega)
\end{tikzcd}
\]
also commutes, and the Universal Property of Inverse Limits yields a unique map
\(\hat{\pi}_1(\Omega)\rightarrow\pi_1^\textnormal{et}(\Omega)\).\\
\indent For injectivity, suppose that \(\hat{g}\in\hat{\pi}_1(\Omega)\) with \(\hat{g}\mapsto 1\).
Then, by Theorem \ref{thm:covers} and Proposition \ref{prop:profinite-deck-transformation-isomorphic},
given any finite, regular cover \(\Omega'=\varprojlim_{\sigma_i'}\Gamma_i'\) of \(\Omega\),
for sufficiently large \(n\in\mathbb{N}\), for all
\(i\in\mathbb{N}_{\geq n}\), \(\textnormal{Deck}(\Omega'/\Omega)\cong\textnormal{Deck}(\Gamma_i'/\Gamma_i)\).
By assumption, the corresponding deck transformation in \(\textnormal{Deck}(\Omega'/\Omega)\) is \(1\),
thus as are, for each \(i\in\mathbb{N}_{\geq n}\), the corresponding deck transformations in each
\(\textnormal{Deck}(\Gamma_i'/\Gamma_i)\). Since this holds for all finite, regular covers of \(\Omega\), \(\hat{g}=1\).\\
\indent For surjectivity, due to the Universal Property of Inverse Limits, it suffices to show that,
for all finite, regular covers \(\Omega'\) of \(\Omega\), the composition
\(\hat{\pi}_1(\Omega)\rightarrow\pi_1^\textnormal{et}(\Omega)\twoheadrightarrow\textnormal{Deck}(\Omega'/\Omega)\)
is surjective. As before, by Theorem \ref{thm:covers} and Proposition \ref{prop:profinite-deck-transformation-isomorphic},
we can write \(\Omega'=\varprojlim_{\sigma_i'}\Gamma_i'\) of \(\Omega\), so that for sufficiently large
\(n\in\mathbb{N}\) and for all \(i\in\mathbb{N}_{\geq n}\), \(\textnormal{Deck}(\Gamma_i'/\Gamma_i)\cong\textnormal{Deck}(\Omega'/\Omega)\).
The surjection \(\hat{\pi}_i(\Omega)\rightarrow\textnormal{Deck}(\Gamma_i'/\Gamma_i)\) then induces a surjection
\(\hat{\pi}_1(\Omega)\rightarrow\textnormal{Deck}(\Omega'/\Omega)\).\\
\indent The case when each \(\Gamma_i=\Gamma\) and \(\sigma_i=\sigma\) is just restricting the paths of
\(\hat{\pi_1}(\Omega)\) to periodic ones using Proposition \ref{prop:elements-projlim}.
\end{proof}
\end{theorem}

\begin{cor}
\(\hat{\pi}_1(\Omega)\) is a topological invariant.
\end{cor}

\begin{remark}
  One can also attempt to construct the inverse map \(\pi_1^\textnormal{et}(\Omega)\rightarrow\hat{\pi}_1(\Omega)\).
  However, a priori, it is unclear that this map is independent of the choice of the specific inverse limit
  structure of \(\Omega\), and therefore is well-defined. One would then need to first explain the connection
  between \(\hat{\pi}_1(\Omega)\) of different inverse limit structures and show that it is a topological invariant,
  thus independent of the choice of inverse limit structures, whereas here, we obtain topological invariance as a consequence.\\
  \indent It is reasonable that this connection between different inverse limit structures is given by \emph{eventual telescoping}.
  That is, up to forgiving a finite initial chunk of two inverse limit structures, their corresponding Borel--Bratteli diagrams
  will be telescope equivalent.
\end{remark}

\begin{remark}
  From Theorem \ref{thm:covers}, one imagines \(\pi_1^\textnormal{et}(\Omega)\) as defined from two inverse limits,
  the first from writing deck transformations of \(\Omega\) as an inverse limit (which can be written as a direct limit
  from Theorem \ref{thm:covers} using Proposition \ref{prop:profinite-deck-transformation-isomorphic}),
  the second from the induced map on deck transformations between finite, regular subcovers.
  \(\hat{\pi}_1(\Omega)\) is also defined using two inverse limits, the first from the induced map on deck
  transformations between finite, regular subcovers of deck transformations, the second from writing deck
  transformations of \(\Omega\) as an inverse limit. Thus the moral proof of Theorem \ref{thm:profinite-to-etale-isomorphism}
  is that all of the squares of the corresponding commutative diagram formed from both inverse limits commute.
  Unfortunately, one can show using the Fibonacci tiling space that these squares are not all of the same fixed size. In fact,
  the sizes of the squares may be unbounded as one moves up in indices along either inverse limit!
\end{remark}

\begin{remark}
  While \(\pi_1^\textnormal{et}(\Omega)\) and \(\hat{\pi}_1(\Omega)\) are isomorphic, the elements of \(\hat{\pi}_1(\Omega)\)
  yield slightly more information than \(\pi_1^\textnormal{et}(\Omega)\), in that the paths within the same path space
  recover the possible covers, not just as topological spaces, but with the allowed inverse limit structures.
\end{remark}

\begin{cor}[Criterion for a finite deck transformation group]
\label{cor:deck-transformation-criterion}
If a finite group \(G\) is a deck transformation group, then there is some \(i\in\mathbb{N}\) and
\(N\trianglelefteq G\) so that \(G=\pi_1(\Gamma_i)/N\). The converse is true if there exists an
\(n\in\mathbb{N}\) such that for all \(i\in\mathbb{N}_{\geq n}\), each
\({{\sigma_{i+1}}_\ast}_N:\pi_1(\Gamma_{i+1})/{\sigma_{i+1}}_\ast^{-1}(N)\rightarrow\pi_1(\Gamma_i)/N\) is an isomorphism.\\
\indent In particular, if each \(\Gamma_i=\Gamma\) and \(\sigma_i=\sigma\), then all finite deck transformation groups are
of the form \(\pi_1(\Gamma)/N\) so that \({\sigma_\ast}_N:\pi_1(\Gamma)/\sigma_\ast^{-1}(N)\rightarrow\pi_1(\Gamma)/N\) is an isomorphism.
\end{cor}

\noindent Although instructive, we leave the example of the dyadic solenoid as an exercise as it is not a tiling space,
and instead illustrate the presentation of elements in \(\hat{\pi}_1(\Omega)\) and demonstrate the identification between
\(\pi_1^\textnormal{et}(\Omega)\) and \(\hat{\pi}_1(\Omega)\) for the Fibonacci tiling space in the next section.

\section{The Fibonacci tiling}\label{sec:Fib}

In this section we study the Fibonacci tiling obtained from the substitution $\sigma_\textnormal{Fib}(a)=ab$, $\sigma_\textnormal{Fib}(b)=a$
and prove Theorem \ref{thm:Fib}. In Subsection \ref{subsec:Fib1} we show that $\hat \pi_1(\Omega_\textnormal{Fib})=\hat{\mathbb{F}}_2$ 
and that there are covers for any 2-generator group $G$ corresponding to any ordered pair of generators. 
By Theorem \ref{thm:substitution-covers}, these covers are all substitution tilings and we show how
to construct $\tilde \sigma$ from a power of $\sigma_\textnormal{Fib}$. 
In Subsection \ref{subsec:Fib3} we examine the spectral properties of the
covers $\tilde \Omega$ and show that the spectrum is pure point if and only if the group $G$ is abelian. 

\subsection{Classifying covers of Fibonacci} \label{subsec:Fib1} Using Barge--Diamond collaring, we can express 
$\Omega_\textnormal{Fib}$ as the inverse limit of the first complex $\Gamma^\textnormal{BD}$ shown in Figure \ref{fig:Fibapproximants}. This 
is homotopy equivalent to the second approximant $\Gamma^\textnormal{AP}$. For calculations involving $\pi_1$, which is 
an invariant of homotopy type, we can do all of our calculations using the simpler space $\Gamma^\textnormal{AP}$. 

\begin{figure}[t]
\centering
\begin{tikzpicture}
\draw[->](-1.1250,.1005)arc[x radius=.75,y radius=.75,start angle=-60,end angle=90];
\draw(-1.5000,1.5000)arc[x radius=.75,y radius=.75,start angle=90,end angle=240];
\draw[->](-1.2682,-.0621)arc[x radius=0.4635,y radius=0.4635,start angle=60,end angle=-90];
\draw(-1.5000,-.9270)arc[x radius=0.4635,y radius=0.4635,start angle=270,end angle=120];
\draw(-1.1250,.1005)--(-1.7317,-.0621);
\fill[white](-1.5000,0)circle[radius=0.1];
\draw(-1.8750,.1005)--(-1.2682,-.0621);
\draw(-1.8750,.1005)--(-1.1250,.1005);
\draw(-1.5000,1.5000)node[above]{\(a\)};
\draw(-1.5000,-.9270)node[below]{\(b\)};
\draw(-1.5000,2.2500)node{\(\Gamma^{\textnormal{BD}}\)};
\draw[->](1.5000,0)arc[x radius=.75,y radius=.75,start angle=-90,end angle=90];
\draw(1.5000,1.5000)arc[x radius=.75,y radius=.75,start angle=90,end angle=270];
\draw[->](1.5000,0)arc[x radius=0.4635,y radius=0.4635,start angle=90,end angle=-90];
\draw(1.5000,-.9270)arc[x radius=0.4635,y radius=0.4635,start angle=270,end angle=90];
\draw(1.5000,1.5000)node[above]{\(a\)};
\draw(1.5000,-.9270)node[below]{\(b\)};
\draw(1.5000,2.2500)node{\(\Gamma^{\textnormal{AP}}\)};
\end{tikzpicture}
\caption{The Barge--Diamond and Anderson--Putnam complexes for $\Omega_\textnormal{Fib}$.
  Since these approximants are homotopy equivalent, we can do our calculations with either one. }
\label{fig:Fibapproximants}
\end{figure}
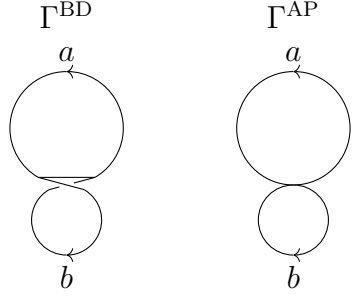%

The fundamental group of $\Gamma^\textnormal{AP}$ is the free group $\mathbb{F}_2$, with generators $a$ and $b$ corresponding to
the two loops. Substitution sends $a$ to $ab$ and $b$ to $a$, which is an isomorphism on $\mathbb{F}_2$. This induces an
isomorphism on $\hat{\mathbb{F}}_2$, so 
\[ \hat \pi_1(\Omega_\textnormal{Fib}) = \ilim \hat{\mathbb{F}}_2 = \hat{\mathbb{F}}_2. \]
Likewise, 
\[ \dlim \textnormal{Sur}(\pi_1(\Gamma^\textnormal{AP}), G) = \dlim \textnormal{Sur}(\mathbb{F}_2, G) = \textnormal{Sur}(\mathbb{F}_2, G). \]
By Theorem \ref{thm:covers}, we have a regular cover of $\Omega_\textnormal{Fib}$ with group $G$ for every surjection 
$\mathbb{F}_2 \to G$, in other words for every ordered pair of generators of $G$. (In particular, $G$ can be any finite
group with up to two generators.) Two such covers are equivalent if and
only if they are related by an automorphism of $G$. 

Specifically, if $G$ is any finite group with generators $(g_1, g_2)$, we can build a covering space
from the following rules: 
\begin{itemize}
\item If a letter $a$ has subscript $g$, the next letter has subscript $gg_1$. 
\item If a letter $b$ has subscript $g$, the next letter has subscript $gg_2$. 
\end{itemize}
That is, the cover is build from the map $\phi$ sending $a$ to $g_1$ and $b$ to $g_2$. 

\begin{example}[Double covers]\label{ex:Fib-double}
There are exactly three covers with group $G=\Z/2\Z = \{0,1\}$ (with additive notation), namely:
\begin{itemize}
\item $\tilde \Omega_1$, with $(g_1,g_2) = (1,0)$, 
\item $\tilde \Omega_2$, with $(g_1,g_2) = (0, 1)$, and 
\item $\tilde \Omega_3$, with $(g_1,g_2) = (1, 1)$.
\end{itemize} 
There are no nontrivial automorphisms of $\Z/2\Z$, so these covers are all inequivalent.  

Since $G$ is abelian, we can also view this as a cohomology problem:
\[ \check H^1(\Omega_\textnormal{Fib}, \Z/2\Z) = \dlim H^1(\Gamma^\textnormal{AP},\Z/2\Z) = H^1(\Gamma^\textnormal{AP},\Z/2\Z) = (\Z/2\Z)^2\]
$(\Z/2\Z)^2$ has three non-trivial elements, each of which defines a double cover of $\Omega_\textnormal{Fib}$, and a
trivial element that defines a single cover. 

Let $p_i$ (with $i=1$, 2, or 3) be the projection of $\tilde \Omega_{i}$ to $\Omega_\textnormal{Fib}$
and let $\phi_i$ be the corresponding map from $\F_2$ to $\Z/2\Z$. Note that $\phi_1\circ \sigma_\textnormal{Fib} = \phi_3$,
since $\phi_1\circ \sigma_\textnormal{Fib}(a) = \phi_1(ab)=1$ and $\phi_1\circ \sigma_\textnormal{Fib}(b) = \phi_1(a)=1$. Similarly, 
$\phi_1 \circ \sigma_\textnormal{Fib}^2 = \phi_2$ and $\phi_1 \circ \sigma_\textnormal{Fib}^3 = \phi_1$. This means that there are
maps $\tilde \Omega_1 \to \tilde \Omega_2 \to \tilde \Omega_3 \to \tilde \Omega_1 $ that cover the substitution map
$\sigma_\textnormal{Fib}$ on $\Omega$. It also means that there are maps $\tilde \sigma_i$ 
from each $\tilde \Omega_i$ to itself that cover $\sigma_\textnormal{Fib}^3$.

To get those maps, we start with 
\[ \sigma_\textnormal{Fib}^3(a) = abaab, \qquad \sigma_\textnormal{Fib}^3(b)=aba \]
and add subscripts. For $\tilde \sigma_i$, we begin each 
supertile with the same subscript as its parent and 
use the following rules for $p_i$ to compute the rest of the subscripts:
\begin{eqnarray*} & \tilde \sigma_1(a_0) = a_0b_1a_1a_0b_1, & \tilde \sigma_1(a_1)=a_1b_0a_0a_1b_0, \cr 
& \tilde \sigma_1(b_0) = a_0b_1a_1, & \tilde \sigma_1(b_1) = a_1b_0a_0, \cr \cr 
& \tilde \sigma_2(a_0) = a_0b_0a_1a_1b_1, & \tilde \sigma_2(a_1)=a_1b_1a_0a_0b_0, \cr 
& \tilde \sigma_2(b_0) = a_0b_0a_1, & \tilde \sigma_2(b_1) = a_1b_1a_0, \cr \cr 
& \tilde \sigma_3(a_0) = a_0b_1a_0a_1b_0, & \tilde \sigma_3(a_1)=a_1b_0a_1a_0b_1, \cr 
& \tilde \sigma_3(b_0) = a_0b_1a_0, & \tilde \sigma_3(b_1) = a_1b_0a_1.
\end{eqnarray*}
\end{example} 

\begin{example}[Triple covers]\label{ex:Fib-triple} When $G=\Z/3\Z$, there are eight non-trivial choices for $(g_1, g_2)$.
Multiplication by $-1=2$ is an automorphism on $G$, so these eight covers only
represent four equivalence classes. $\left (\begin{smallmatrix} 1&1 \cr 1&0 \end{smallmatrix} \right )$ has order
8 in $GL_2(\Z/3\Z)$, but the fourth power is $\left (\begin{smallmatrix} 2&0 \cr 0&2 \end{smallmatrix} \right )$,
which implements the automorphism on $(\Z/3\Z)^2$. If we wanted, we could ignore the automorphism and construct 
the substitution rules for our covers from $\sigma_\textnormal{Fib}^8$ and our following rules, exactly as before, 
but it is simpler to use $\sigma_\textnormal{Fib}^4$ and the automorphism, for instance 
\be \label{eq:shortcut} a_g \to a_{2g}b_{2g+1}a_{2g+1}a_{2g+2}b_{2g}a_{2g}b_{2g+1}a_{2g+1}, \qquad 
b_g \to a_{2g}b_{2g+1}a_{2g+1}a_{2g+2}b_{2g} \ee
for the cover with $(g_1,g_2)=(1,0)$. 

More generally, suppose we have a $G$-covering map $p$ generated by a map $\phi: \F_2 \to G$ and that
$p \circ \sigma^n$ is equivalent to $A \circ \phi =  \phi \circ \sigma_\textnormal{Fib}^n$.
We then define the supertiles $\tilde \sigma(a_g)$
and $\tilde \sigma(b_g)$ by writing $\sigma_\textnormal{Fib}^k$ of $a$ or $b$, setting the subscript of the first letter
equal to $A(g)$, and using the following rules from $\phi$ to determine the remaining letters. 

\end{example}

\begin{example}[A non-abelian cover]\label{ex:non-abelian} 
Let $G=S_3$ be the group of permutations of 
$\{1,2,3\}$. Up to conjugation, there are three ways to choose a pair of generators, 
depending on whether each generator is odd (an interchange of two elements) or even
(a cyclic permutation of three elements). If $a$ and $b$ map to distinct odd exchanges, 
then $\sigma(b)=a$ maps to an exchange, $\sigma^2(b)=\sigma(a)=ab$ maps to a cyclic permutation, 
$\sigma^3(b)=\sigma^2(a)=aba$ maps to an exchange, and $\sigma^3(a)=abaab=a$ maps to an exchange. That
is, 
\begin{eqnarray}
\phi_0(a,b) & = & (\hbox{exchange}, \hbox{exchange}), \cr 
\phi_0\circ \sigma(a,b) & = & (\hbox{cyclic}, \hbox{exchange}), \cr 
\phi_0\circ \sigma^2(a,b) & = & (\hbox{exchange}, \hbox{cyclic}), \cr
\phi_0\circ \sigma^3(a,b) & = & (\hbox{exchange}, \hbox{exchange}).
\end{eqnarray}
Substitution permutes the three possibilities and any surjection 
$\phi_0: \mathbb{F}_2 \to G$ is equivalent to $\phi_0 \circ \sigma^3$, with
$\phi_0 \circ \sigma^3 = A \circ \phi_0$ for some automorphism $A$. 

We can thus understand
all $G$-covers of $\Omega_\textnormal{Fib}$ by studying the single case where $a$ and $b$ are sent
to distinct exchanges that, in an abuse of notation, we will denote $a$ and $b$. 
In this case, that autmorphism $A$ is conjugation by $\phi_0(a)$ and $\tilde \Omega$ is
generated by the substitution 
\begin{eqnarray} 
a_g & \mapsto & a_{aga} b_{ag} a_{agb} a_{agba} b_{agb}, \cr 
b_g & \mapsto & a_{aga} b_{ag} a_{agb}. 
\end{eqnarray} 
\end{example}

\subsection{The profinite perspective}\label{subsec:Fib3}

We revisit the double covers of Example \ref{ex:Fib-double} from the profinite perspective, examining
$\hat \pi_1(\Omega)$ and its index-2 quotients. Beginning with $\pi_1(\Gamma_0)=\F_2=\langle a,b \rangle$, 
we examine the index-2 quotient \(G_0=\left\langle a,b\left| a^2,ab\right.\right\rangle\). 
The associated path of quotients of $\pi_1(\Gamma_n)$, involves
\begin{align*}
G_1&=\left\langle a,b\left| \sigma_\ast^{-1}(a^2),\sigma_\ast^{-1}(ab)\right.\right\rangle=\left\langle a,b\left| b^2,bb^{-1}a\right.\right\rangle=\left\langle a,b\left| a,b^2\right.\right\rangle \\
G_2&=\left\langle a,b\left| \sigma_\ast^{-1}(a),\sigma_\ast^{-1}(b^2)\right.\right\rangle=\left\langle a,b\left| b,(b^{-1}a)^2\right.\right\rangle=\left\langle a,b\left| a^2,b\right.\right\rangle \\
G_3&=\left\langle a,b\left| \sigma_\ast^{-1}(a^2),\sigma_\ast^{-1}(b)\right.\right\rangle=\left\langle a,b\left| b^2,b^{-1}a\right.\right\rangle=\left\langle a,b\left| a^2,ab\right.\right\rangle = G_0,
\end{align*}
so the path is \(3\)-periodic, with all groups being index-\(2\) quotients of \(\pi_1(\Gamma_0)\). 

\begin{figure}[t]
\centering
\begin{tikzpicture}
\draw(-4.5,0)--(0.5,0);
\draw(0,-0.5)--(0,2.5);
\draw(0.5,0)node[right]{\(i\)};
\draw(0,-0.5)node[below]{\(n\)};
\draw(-4,-0.1)--(-4,0.1);
\draw(-4,-0.2)node[below]{\(4\)};
\draw(-3,-0.1)--(-3,0.1);
\draw(-3,-0.2)node[below]{\(3\)};
\draw(-2,-0.1)--(-2,0.1);
\draw(-2,-0.2)node[below]{\(2\)};
\draw(-1,-0.1)--(-1,0.1);
\draw(-1,-0.2)node[below]{\(1\)};
\draw(-0.1,0.5)--(0.1,0.5);
\draw(0.1,0.5)node[right]{\(1\)};
\draw(-0.1,1)--(0.1,1);
\draw(0.1,1)node[right]{\(\left\langle a,b\left| a^2,ab\right.\right\rangle\)};
\draw(-0.1,1.5)--(0.1,1.5);
\draw(0.1,1.5)node[right]{\(\left\langle a,b\left| a,b^2\right.\right\rangle\)};
\draw(-0.1,2)--(0.1,2);
\draw(0.1,2)node[right]{\(\left\langle a,b\left| a^2,b\right.\right\rangle\)};
\fill(0,0.5)circle[radius=0.05];
\fill(-1,0.5)circle[radius=0.05];
\fill(-2,0.5)circle[radius=0.05];
\fill(-3,0.5)circle[radius=0.05];
\fill(-4,0.5)circle[radius=0.05];
\fill(0,1)circle[radius=0.05];
\fill(-1,1)circle[radius=0.05];
\fill(-2,1)circle[radius=0.05];
\fill(-3,1)circle[radius=0.05];
\fill(-4,1)circle[radius=0.05];
\fill(0,1.5)circle[radius=0.05];
\fill(-1,1.5)circle[radius=0.05];
\fill(-2,1.5)circle[radius=0.05];
\fill(-3,1.5)circle[radius=0.05];
\fill(-4,1.5)circle[radius=0.05];
\fill(0,2)circle[radius=0.05];
\fill(-1,2)circle[radius=0.05];
\fill(-2,2)circle[radius=0.05];
\fill(-3,2)circle[radius=0.05];
\fill(-4,2)circle[radius=0.05];
\draw(-4,0.5)--(0,0.5);
\draw(-4,1)--(-3,2)--(-2,1.5)--(-1,1)--(0,2);
\draw(-4,1.5)--(-3,1)--(-2,2)--(-1,1.5)--(0,1);
\draw(-4,2)--(-3,1.5)--(-2,1)--(-1,2)--(0,1.5);
\end{tikzpicture}
\caption{Drawing the inverse limit horizontally and the deck transformation groups vertically, we see that the Fibonacci
  substitution has finite deck transformation groups forming Borel--Bratteli diagrams consisting of singleton path spaces,
  with each component (here four are drawn, \(1\), \(\left\langle a,b\left| a^2,ab\right.\right\rangle\),
  \(\left\langle a,b\left| a,b^2\right.\right\rangle\), and \(\left\langle a,b\left| a^2,b\right.\right\rangle\))
  giving rise to a distinct part of \(\hat{\pi}_1(\Omega)\), and each path is minimal.}
\label{fig:fibonacci}
\end{figure}
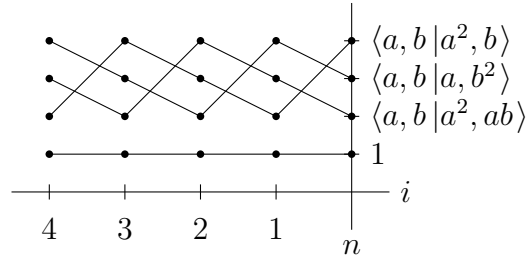%

The three quotients $G_i$ correspond to the three (regular) double covers of \(S^1\vee S^1\). The
double cover of $\Omega_\textnormal{Fib}$ defined by $G_0$ is thus associated with a sequence of 
double covers of the approximants $\Gamma_i$, as shown in Figure \ref{fig:cover-of-approximants}. If we
shift the diagram, we obtain the covers defined by $G_1$ and $G_2$. The three maps 
\begin{align*}
\sigma_1'(a_0)&=a_0b_1, &\sigma_2'(a_0)&=a_0b_0, &\sigma_3'(a_0)&=a_0b_1, \\
\sigma_1'(a_1)&=a_1b_0, &\sigma_2'(a_1)&=a_1b_1, &\sigma_3'(a_1)&=a_1b_0, \\
\sigma_1'(b_0)&=a_0, &\sigma_2'(b_0)&=a_0, &\sigma_3'(b_0)&=a_0, \\
\sigma_1'(b_1)&=a_1, &\sigma_2'(b_1)&=a_1, &\sigma_3'(b_1)&=a_1, 
\end{align*}
can be understood either as maps between approximants to the covering space $\tilde \Omega$
or as maps between different covering spaces $\tilde \Omega_i$ that cover the substitution map on 
$\Omega_\textnormal{Fib}$. The maps $\sigma_i$ of Example \ref{ex:Fib-double} that cover 
$\sigma_\textnormal{Fib}^3$ are essentially the compositions $\sigma_1'\circ \sigma_2' \circ \sigma_3'$,
$\sigma_2' \circ \sigma_3' \circ \sigma_1'$, and $\sigma_3' \circ \sigma_1' \circ \sigma_2'$.

\begin{figure}[t]
\centering
\begin{tikzpicture}
\draw(-2.0625,0)node{\(\cdots\)};
\draw[->](-1.6875,0)--(-.9375,0);
\draw[->](0,-1.5000)arc[x radius=.75,y radius=1.5,start angle=-90,end angle=90];
\draw[->](0,1.5000)arc[x radius=.75,y radius=1.5,start angle=90,end angle=270];
\draw[->](0,-.7500)arc[radius=.75,start angle=-90,end angle=90];
\draw[->](0,.7500)arc[radius=.75,start angle=90,end angle=270];
\fill(-.7500,0)circle[radius=0.05];
\draw(0,-.7500)node[above]{\(a_0\)};
\draw(0,-1.5000)node[above]{\(b_0\)};
\draw(0,.7500)node[below]{\(a_1\)};
\draw(0,1.5000)node[below]{\(b_1\)};
\draw[->](.9375,0)--(1.6875,0)node[midway,above]{\(\sigma_3'\)};
\draw[->](2.6250,-2.2500)arc[radius=.75,start angle=-90,end angle=270];
\draw[->](1.8750,0)arc[radius=.75,start angle=180,end angle=360];
\draw[->](3.3750,0)arc[radius=.75,start angle=0,end angle=180];
\draw[->](2.6250,2.2500)arc[radius=.75,start angle=90,end angle=450];
\fill(2.6250,-.7500)circle[radius=0.05];
\draw(3.3750,0)node[left]{\(a_0\)};
\draw(2.6250,-2.2500)node[above]{\(b_0\)};
\draw(1.8750,0)node[right]{\(a_1\)};
\draw(2.6250,2.2500)node[below]{\(b_1\)};
\draw[->](3.5625,0)--(4.3125,0)node[midway,above]{\(\sigma_2'\)};
\draw[->](5.2500,-2.2500)arc[radius=.75,start angle=-90,end angle=270];
\draw[->](4.5000,0)arc[radius=.75,start angle=180,end angle=360];
\draw[->](6.0000,0)arc[radius=.75,start angle=0,end angle=180];
\draw[->](5.2500,2.2500)arc[radius=.75,start angle=90,end angle=450];
\fill(5.2500,-.7500)circle[radius=0.05];
\draw(5.2500,-2.2500)node[above]{\(a_0\)};
\draw(6.0000,0)node[left]{\(b_0\)};
\draw(5.2500,2.2500)node[below]{\(a_1\)};
\draw(4.5000,0)node[right]{\(b_1\)};
\draw[->](6.1875,0)--(6.9375,0)node[midway,above]{\(\sigma_1'\)};
\draw[->](7.8750,-1.5000)arc[x radius=.75,y radius=1.5,start angle=-90,end angle=90];
\draw[->](7.8750,1.5000)arc[x radius=.75,y radius=1.5,start angle=90,end angle=270];
\draw[->](7.8750,-.7500)arc[radius=.75,start angle=-90,end angle=90];
\draw[->](7.8750,.7500)arc[radius=.75,start angle=90,end angle=270];
\fill(7.1250,0)circle[radius=0.05];
\draw(7.8750,-.7500)node[above]{\(a_0\)};
\draw(7.8750,-1.5000)node[above]{\(b_0\)};
\draw(7.8750,.7500)node[below]{\(a_1\)};
\draw(7.8750,1.5000)node[below]{\(b_1\)};
\draw[->](0,-1.6875)--(0,-3.1875);
\draw[->](2.6250,-2.4375)--(2.6250,-3.1875);
\draw[->](5.2500,-2.4375)--(5.2500,-3.1875);
\draw[->](7.8750,-1.6875)--(7.8750,-3.1875);
\draw(-2.0625,-4.8750)node{\(\cdots\)};
\draw[->](-1.6875,-4.8750)--(-.9375,-4.8750);
\draw[->](0,-3.3750)arc[radius=.75,start angle=90,end angle=450];
\draw[->](0,-6.3750)arc[radius=.75,start angle=-90,end angle=270];
\fill(0,-4.8750)circle[radius=0.05];
\draw(0,-6.3750)node[above]{\(a\)};
\draw(0,-3.3750)node[below]{\(b\)};
\draw[->](.9375,-4.8750)--(1.6875,-4.8750)node[midway,above]{\(\sigma\)};
\draw[->](2.6250,-3.3750)arc[radius=.75,start angle=90,end angle=450];
\draw[->](2.6250,-6.3750)arc[radius=.75,start angle=-90,end angle=270];
\fill(2.6250,-4.8750)circle[radius=0.05];
\draw(2.6250,-6.3750)node[above]{\(a\)};
\draw(2.6250,-3.3750)node[below]{\(b\)};
\draw[->](3.5625,-4.8750)--(4.3125,-4.8750)node[midway,above]{\(\sigma\)};
\draw[->](5.2500,-3.3750)arc[radius=.75,start angle=90,end angle=450];
\draw[->](5.2500,-6.3750)arc[radius=.75,start angle=-90,end angle=270];
\fill(5.2500,-4.8750)circle[radius=0.05];
\draw(5.2500,-6.3750)node[above]{\(a\)};
\draw(5.2500,-3.3750)node[below]{\(b\)};
\draw[->](6.1875,-4.8750)--(6.9375,-4.8750)node[midway,above]{\(\sigma\)};
\draw[->](7.8750,-3.3750)arc[radius=.75,start angle=90,end angle=450];
\draw[->](7.8750,-6.3750)arc[radius=.75,start angle=-90,end angle=270];
\fill(7.8750,-4.8750)circle[radius=0.05];
\draw(7.8750,-6.3750)node[above]{\(a\)};
\draw(7.8750,-3.3750)node[below]{\(b\)};
\end{tikzpicture}
\caption{The cover of an inverse limit can be seen as a sequence of covers of the corresponding
approximants}\label{fig:cover-of-approximants}
\end{figure}
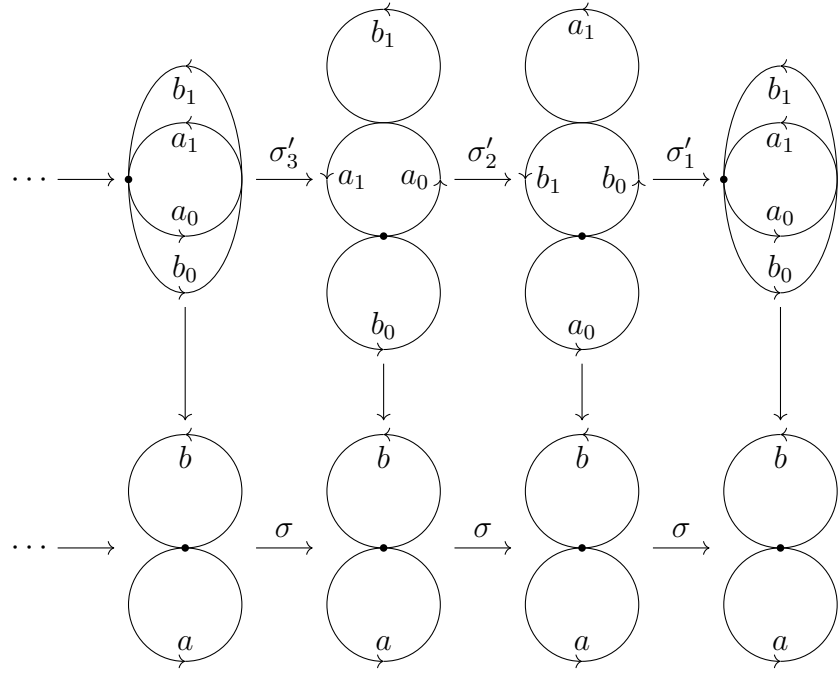

\subsection{Maximal equicontinuous factors}\label{subsec:Fib4}

Let $\Omega_\textnormal{Fib}$ be the Fibonacci tiling space, where we suppose that the $a$ and $b$ tiles have lengths 
$|a|$ and $|b|$. (The lengths do not matter, as different choices
of lengths give conjugate dynamics, up to an overall scale.) The maximal equicontinuous factor 
$M_\textnormal{Fib}$ of $\Omega_\textnormal{Fib}$ is the 2-torus $\R^2/\Z^2$. Translation in $\Omega_\textnormal{Fib}$ corresponds 
to an irrational winding in $M_\textnormal{Fib}$, with slope $\phi$. The projection $\pi: \Omega_\textnormal{Fib} \to M_\textnormal{Fib}$ is 
almost everywhere one-to-one, making $\Omega_\textnormal{Fib}$ and $M_\textnormal{Fib}$ measurably conjugate dynamical systems. In 
particular, $\Omega_\textnormal{Fib}$ has pure point spectrum, with eigenfunctions that are pullbacks of the eigenfunctions
on $M_\textnormal{Fib}$ for the irrational winding. These are the Fourier modes $\psi_{m,n}(x,y) = \exp(2 \pi i(mx+ny))$
with eigenvalues $c(m+\phi n)$, 
where $c^{-1} = (\phi|a| +|b|)/\phi^2$ is the average length of a tile and $\phi = (1+\sqrt{5})/2$ is the golden
mean.

If $\tilde \Omega$ is a $G$-cover of $\Omega_\textnormal{Fib}$, then we must look at the maximal equicontinuous factor 
$\tilde M$ of $\tilde \Omega$. $\tilde M$ is itself a cover of $M_\textnormal{Fib}$ and we have a commutative diagram
\[
\begin{tikzcd}
\tilde{\Omega}\arrow[r,two heads]\arrow[d,two heads]&\Omega_\textnormal{Fib}\arrow[d,two heads]\\
\tilde{M}\arrow[r,two heads]&M_\textnormal{Fib}
\end{tikzcd}
\]
where the vertical arrows are factor maps and the horizontal arrows are covering maps, with the group of 
deck transformations for $\tilde M \to M_\textnormal{Fib}$ being a quotient of $G$. 
There are two possibilities: 

\begin{itemize}
\item If $\tilde M$ is a $G$-cover of $M_\textnormal{Fib}$, then the factor map $\tilde \Omega \to \tilde M$ is almost 
everywhere one-to-one, so $\tilde \Omega$ has pure point spectrum.
\item If $\tilde M$ is not a $G$-cover, then the factor map $\tilde \Omega \to \tilde M$ is almost everywhere 
multiple-to-one, so $L^2(\tilde \Omega)$ is not the pullback of $L^2(\tilde M)$, so $\tilde \Omega$ has mixed spectrum,
with eigenfunctions that are pullbacks of functions on $\tilde M$ and continuous spectrum that does not 
come from $\tilde M$. 
\end{itemize}

$M_\textnormal{Fib}$ is a torus, so the covers of $M_\textnormal{Fib}$ correspond to finite-index subgroups of 
$\pi_1(M_\textnormal{Fib})=\Z^2$. Since $\Z^2$ is abelian, all covers are normal and have an abelian group of deck 
transformations. In particular, if $G$ is non-abelian, then $\tilde M$ cannot be a $G$-cover of $M_\textnormal{Fib}$ 
and $\tilde \Omega$ cannot have pure point spectrum. 

On the other hand, if $G$ is abelian then we claim that $\tilde M$ is indeed a $G$-cover of $M_\textnormal{Fib}$.
To see this, we suppose that the tile lengths are $|a|=\phi$ and $|b|=1$ 
(all other cases being topologically conjugate to uniform rescalings of this case). 
The maximal equicontinuous factors $M_\textnormal{Fib}$ and $\tilde M$ are built from the eigenvalues of 
translation, which in turn are based on return vectors for large patches. We must show that
the return lattice for patches in $\tilde \Omega$ is an order-$|G|$ sublattice of the 
return lattice for patches in $\Omega_\textnormal{Fib}$. 

We can assume, without loss of generality, that $\tilde \Omega$ is constructed using 
subscripts and following rules based on a specific surjection $\phi_0: \mathbb{F}_2 \to G$. 
Since $G$ is abelian, $\phi_0$ factors through
$\Z^2$. Let $L \subset \Z^2$ be 
the abelianization of the kernel of $\phi_0$, so that $G \simeq \Z^2/L$. 
By construction, the abelianization of any return word for $\tilde \Omega$ lies in $L$,
so the return lattice for $\tilde \Omega$ has index at least $|G|$ in $\Z^2$. 
Since it cannot have index greater than $|G|$, the return lattice must be exactly $L$ and 
$\tilde M = \R^2/L$ must be a regular $G$-cover of $M_\textnormal{Fib} = \R^2/\Z^2$. 

This completes the proof of Theorem \ref{thm:Fib}.

Looking back at Example \ref{ex:non-abelian}, the kernel of $\phi_0$ was generated by $a^2$, $b^2$ and $(ab)^3$,
whose abelianizations are $(2,0)$, $(0,2)$ and $(3,3)$. The sublattice of $\Z^2$ generated by the 
abelianization of the kernel of $\phi_0$ was thus the index-2 lattice generated by 
$(1,1)$ and $(2,0)$. $\tilde M$ was a double cover of $M_\textnormal{Fib}$, but not a $G$-cover, so 
$\tilde \Omega$ did not have pure point spectrum.

\section{Other examples}\label{sec:examples} 

In this section, we present a number of illustrative examples in one and two dimensions. All are 
substitution tilings, as those spaces have the simplest inverse limit structures. Further 
background on many of these can be found in \cite{TAO}.

\subsection{One-dimensional examples}

\begin{example}[Period doubling and Thue--Morse] The period doubling and Thue--Morse tiling spaces have
isomorphic first cohomology $\Z[1/2] \oplus \Z$, so the classification of abelian covers of 
$\Omega_\textnormal{TM}$ is identical to the classification of abelian covers of $\Omega_\textnormal{PD}$.

$\check H^1(\Omega_\textnormal{PD}, \Z/2\Z) = \check H^1(\Omega_\textnormal{TM}, \Z/2\Z)= \Z/2\Z$, so there is exactly 
one double cover of $\Omega_\textnormal{PD}$ (namely $\Omega_\textnormal{TM}$) and one double cover of $\Omega_\textnormal{TM}$. 
The maximal equicontinuous factors of $\Omega_\textnormal{PD}$ and $\Omega_\textnormal{TM}$ are the 
dyadic solenoid $\ilim \R/(2^n \Z)$, which does not admit any double covers at all. Thus the maximal 
equicontinuous factor of the cover equals the maximal equicontinuous factor of the base, 
so the coincidence rank of the cover is twice the coincidence rank of the base. 
This explains why $\Omega_\textnormal{TM}$ has mixed spectrum. 

As for triple covers,  
$\check H^1(\Omega_\textnormal{TM},\Z/3\Z) = \check H^1(\Omega_\textnormal{PD},\Z/3\Z) = (\Z/3\Z)^2$. 
This leaves us with 8 non-trivial triple covers of each space. For period doubling, we can 
construct these covers by sending $(a,b)$ to $(0,1)$, $(0,2)$, $(1,0)$, $(1,1)$, $(1,2)$, 
$(2,0)$, $(2,1)$ or $(2,2)$. 
Substitution swaps two of these and permutes the remaining six: 
\begin{eqnarray} 
&& (1,1) \to (2,2) \to (1,1) \cr 
&& (0,1) \to (1,0) \to (1,2) \to (0,2) \to (2,0) \to (2,1) \to (0,1). 
\end{eqnarray} 

The automorphism $1 \leftrightarrow 2$ of $\Z/3\Z$ relates 
$(1,1) \leftrightarrow (2,2)$, $(0,1) \leftrightarrow (0,2)$, $(1,0) \leftrightarrow (2,0)$ and $(1,2) \leftrightarrow (2,1)$.
This leaves four covers, with substitution fixing the $(1,1)$ cover and cyclically permuting the other three. 

The maximal equicontinuous factor of the $(1,1)$ cover is a triple cover of the solenoid
$\ilim \R/(2^n \Z)$, so the $(1,1)$ cover of period-doubling has pure point spectrum. However, the 
maximum equicontinuous factor of the other covers are the solenoid itself, so those covers have mixed
spectrum. 

Covers of Thue--Morse have an almost identical structure. To get the following rules for a triple cover of 
Thue--Morse, first collar Thue--Morse to obtain subscripts $a$ and $b$ and then apply the exact same 
following rules (in terms of $a$ and $b$) as for the analogous cover of period doubling. 

While $\Omega_\textnormal{PD}$ and $\Omega_\textnormal{TM}$ can't be distinguished by their abelian covers, they can
be distinguished by their non-abelian covers. For period-doubling, the complex $\Gamma_n$ is homotopy
equivalent to the wedge of two circles. 
For Thue--Morse, the complex $\Gamma_n$ is shown in Figure \ref{fig:TMapproximant} and is 
homotopy equivalent to the wedge of three circles.

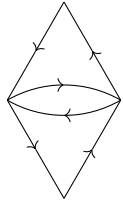
\begin{figure}[t]
\centering
\begin{tikzpicture}
\draw[->](-.7500,0)arc[x radius=1.5,y radius=1.5,start angle=120,end angle=90];
\draw(0,.2009)arc[x radius=1.5,y radius=1.5,start angle=90,end angle=60];
\draw[->](.7500,0)arc[x radius=1.5,y radius=1.5,start angle=-60,end angle=-90];
\draw(0,-.2009)arc[x radius=1.5,y radius=1.5,start angle=-90,end angle=-120];
\draw[->](-.7500,0)--(-.3750,-.6495);
\draw(-.3750,-.6495)--(0,-1.2990);
\draw[->](0,-1.2990)--(.3750,-.6495);
\draw(.3750,-.6495)--(.7500,0);
\draw[->](.7500,0)--(.3750,.6495);
\draw(.3750,.6495)--(0,1.2990);
\draw[->](0,1.2990)--(-.3750,.6495);
\draw(-.3750,.6495)--(-.7500,0);
\end{tikzpicture}
\caption{The collared Anderson--Putnam complex for $\Omega_\textnormal{TM}$.}
\label{fig:TMapproximant}
\end{figure}%

The action of substitution on $\hat \pi_1(\Gamma_n(\Omega_\textnormal{TM}))$ is too complicated to write down a direct limit in closed form.
To understand the difference between period doubling and Thue--Morse, we need to look at a specific 
finite non-abelian group $G$. 
As Erdin \cite{erdin2010patternequivariantrepresentationvariety} first showed, the simplest non-abelian group $G=S_3$ is enough to 
distinguish between 
$\Omega_\textnormal{PD}$ and $\Omega_\textnormal{TM}$. For $\Omega_\textnormal{PD}$, we are picking two generators 
$(g_1, g_2)$ and a map $\phi_n(a)=g_1$ and $\phi_n(b)=g_2$. But no matter what $g_1$ and $g_2$ are, $\phi_n \circ \sigma^2$
is not a surjection, so $\dlim \textnormal{Sur}(\pi_1(\Gamma_{\textnormal{PD},n}),S_3)$ is empty; there are no regular connected covers of 
$\Omega_\textnormal{PD}$ with deck group $S_3$. 

However, $\dlim \textnormal{Sur}(\pi_1(\Gamma_{\textnormal{TM},n},S_3)$ is not empty. If we pick our three generators
to be $\big ((1,2,3), (1,2,3), (1,2)\big )$, then ${\sigma_\textnormal{TM}}_\ast^2$ maps this to 
$\big ((1,3,2), (1,3,2), (1,3)\big )$, which is conjugate to our original triplet of generators. 
Based on this map, we can construct 
a regular connected cover of $\Omega_\textnormal{TM}$ with deck group $S_3$. There are three additional inequivalent elements of 
$\dlim \textnormal{Sur}(\pi_1(\Gamma_{\textnormal{TM},n},S_3)$, given by the generators $\big ((1,2,3), (1,3,2), (1,2)\big )$,
$\big ((1,2,3), (1), (1,2)\big )$, and $\big ((1), (1,2,3), (1,2)\big )$.
\end{example}

\begin{example}[Silver mean] Consider the 1-dimensional substitution $\sigma_\textnormal{sm}(a)=abb$, $\sigma_\textnormal{sm}(ab)$.
This is called the {\em silver mean} substitution since the leading eigenvalue of the substitution is the silver
mean $1+\sqrt{2}$. This substitution has a number of similarities to the 
Fibonacci substitution:
\begin{itemize} 
\item The Anderson--Putnam complex is a figure 8. The Barge--Diamond complex is homotopy equivalent to the 
Anderson--Putnam complex, so $\hat \pi_1(\Gamma_n) = \hat{\mathbb{F}}_2$. 
\item Substitution induces an isomorphism of $\mathbb{F}_2$, since $\sigma_\textnormal{sm}(ba^{-1}b)=a$ and $\sigma_\textnormal{sm}(b^{-1}a)=b$. 
Thus $\ilim \hat \pi_1(\Gamma_n) = \ilim \hat{\mathbb{F}}_2 = \hat{\mathbb{F}}_2$. 
\item There exist regular covers corresponding to all finite groups with two generators. 
\end{itemize} 

Since $\hat \pi_1(\Omega_\textnormal{sm})$ and $\hat \pi_1(\Omega_\textnormal{Fib})$ are isomorphic, we cannot distinguish 
$\Omega_\textnormal{sm}$ and $\Omega_\textnormal{Fib}$ by their covers, by their representation varieties, or by their \v Cech 
cohomology. However, they are not topologically conjugate dynamical systems, not even homeomorphic. 
It is known from \cite{MR1898159} that homeomorphic substitution tiling spaces in dimension \(1\) must have inflation factors $\lambda$
and $\lambda'$ which satisfy $\mathbb{Q}(\lambda)=\mathbb{Q}(\lambda')$. As follows from Perron--Frobenius theory,
this condition is necessary for the respective frequency modules to agree, which are also part of the natural order 
structures on $\check H^1(\Omega) = \Z^2$. Moreover, the two tilings have different point spectra, one being 
a multiple of $\Z[\phi]$ and the other a multiple of $\Z[\sqrt{2}]$.  
\end{example}

\begin{example}
Clearly, the profinite fundamental group cannot distinguish all non-homeomorphic
tiling spaces. The two 1-dimensional tiling spaces generated by the substitutions
 \begin{alignat*}{5}
    \sigma_1 &: a \rightarrow cabc, && \quad b \rightarrow ca, && \quad c \rightarrow b, \\
    \sigma_2 &: a \rightarrow bcac, && \quad b \rightarrow ca, && \quad c \rightarrow b
 \end{alignat*}
have the same (or isomorphic) inflation factor, cohomology (of rank $5$),
ordered cohomology, frequency module, asymptotic composants \cite{BD01,Gah26}, 
and profinite fundamental group (the action of the substitution on $\pi_1(\Gamma_n)$
is invertible in both cases). They do differ in their orbit separation dimension (OSD) \cite{BGG25,Gah26}, 
however, but this is a dynamical invariant, not a purely topological one.
The OSD is the box dimension of the projection of a transversal of $\Omega$ to
the MEF, measured relative to a metric which is dynamically defined \cite{BGG25} 
(although with respect to the usual translation dynamics).
\end{example}

\begin{example}
If the substitution action on $\pi_1(\Gamma_n)$ is not an isomorphism,
we have much better chances to find some deck group $G$, such that
the number of $G$-covers distinguishes our tiling from other, similar
tilings. The two substitution
\begin{alignat*}{3}
    \sigma_3 &: \quad a \rightarrow abbaa, && \quad b \rightarrow aab, \\
    \sigma_4 &: \quad a \rightarrow aaabb, && \quad b \rightarrow aab
\end{alignat*}
have the same substitution matrix as $\sigma_\textnormal{Fib}^3$, and hence the same
cohomology (of rank 2). Their action on $\pi_1(\Gamma_n)$ is not invertible,
however, and the two have non-isomophic profinite fundamental groups,
which differ also from that of Fibonacci. Indeed, both have
only two covers with deck group $S_3$ (Fibonacci has $3$), and the first
has $3$ covers with deck group $S_4$, whereas the second has none.
\end{example}

\subsection{Two or more dimensions}
In more than one dimension, the calculations are much more difficult. 
$\pi_1(\Gamma_n)$ is finitely presented but is usually not a free group.
The profinite fundamental group
$\hat\pi_1(\Omega)$ is an inverse limit constructed from endomorphisms of finitely presented
groups. Even for simple block substitutions, this is difficult to express in closed form. 
As with some difficult one-dimensional dimensional examples, we often need to look at the 
representation varieties of specific groups and to rely on computers. 

A simple setting is that of block substitutions, where the tiles are unit squares 
and the substitution sends each square to an $n \times n$ array of squares for some integer $n$.
Instead of collaring the basic tiles, we can work with a dual tiling whose vertices correspond to the
tiles of the original tiling, whose edges correspond to pairs of adjacent tiles, and whose tiles 
correspond to \(2\times 2\) blocks of original tiles. There is an induced substitution on this
dual tiling and this substitution ``forces the border'', so no further collaring is needed. We do, 
however, include also the Penrose tiling as example, to demonstrate that the restriction to 
block substitutions is only for convenience and ease of understanding.

Next we compute the fundamental group of the
AP complex $\Gamma$ of the dual tiling. Following \cite{MR1867354}, we find a
\emph{spanning tree} of the 1-skeleton of $\Gamma$, a connected set of
edges which forms a maximal tree in the 1-skeleton.
Adding any of the remaining edges (which we call bridges) creates a
loop. The fundamental group of $\Gamma$ is then finitely presented, 
with one generator for each of the bridges and relations based on the boundaries of the 2-cells. 
From the action of substitution on edges and 2-cells, we can compute the action of substitution on 
$\pi_1(\Gamma)$ and $\textnormal{Hom}(\pi_1(\Gamma),G)$ and determine the appropriate inverse and direct limits. 
The entire procedure is algorithmic and can be handled by computer programs that are written to analyze
arbitrary substitutions. 

\begin{example}[The squiral tiling] The squiral tiling \cite{BGG13} is MLD to a $3\times 3$ block substitution
tiling with two types of tiles and the rule 
\[
\begin{tikzpicture}
\draw(0,0)--(.50,0)--(.50,.50)--(0,.50)--(0,0);
\draw(.25,.25)node{\(0\)};
\draw[|->](.750,.250)--(1.250,.250);
\draw(1.5,-.5)--(3,-.5)--(3,1)--(1.5,1)--(1.5,-.5);
\draw(1.5,0)--(3,0);
\draw(1.5,.5)--(3,.5);
\draw(2,-.5)--(2,1);
\draw(2.5,-.5)--(2.5,1);
\draw(1.75,-.25)node{\(1\)};
\draw(1.75,.25)node{\(0\)};
\draw(1.75,.75)node{\(1\)};
\draw(2.25,-.25)node{\(0\)};
\draw(2.25,.25)node{\(0\)};
\draw(2.25,.75)node{\(0\)};
\draw(2.75,-.25)node{\(1\)};
\draw(2.75,.25)node{\(0\)};
\draw(2.75,.75)node{\(1\)};
\draw(4,0)--(4.50,0)--(4.50,.50)--(4,.50)--(4,0);
\draw(4.25,.25)node{\(1\)};
\draw[|->](4.750,.250)--(5.250,.250);
\draw(5.5,-.5)--(7,-.5)--(7,1)--(5.5,1)--(5.5,-.5);
\draw(5.5,0)--(7,0);
\draw(5.5,.5)--(7,.5);
\draw(6,-.5)--(6,1);
\draw(6.5,-.5)--(6.5,1);
\draw(5.75,-.25)node{\(0\)};
\draw(5.75,.25)node{\(1\)};
\draw(5.75,.75)node{\(0\)};
\draw(6.25,-.25)node{\(1\)};
\draw(6.25,.25)node{\(1\)};
\draw(6.25,.75)node{\(1\)};
\draw(6.75,-.25)node{\(0\)};
\draw(6.75,.25)node{\(1\)};
\draw(6.75,.75)node{\(0\)};
\end{tikzpicture}
\]
As with the Thue--Morse tiling, there is an obvious involution interchanging the 0 and 1 tiles. 
That is, $\Omega_\textnormal{SQ}$ is the double cover of another tiling space; that quotient space
has pure point spectrum, being an almost-1:1 extension of its maximal equicontinuous factor.

Working with the dual substitution, $\pi_1(\Gamma)$ turns out to be $\Z^2$, with substitution
cubing (or tripling in additive notation) each generator. This has two interesting consequences. 
\begin{enumerate}
\item Since $\hat \pi_1(\Omega_\textnormal{SQ})$ is abelian, all finite covers of $\Omega_\textnormal{SQ}$
have abelian deck transformation groups. These can all be understood in terms of 
$\check H^1(\Omega_\textnormal{SQ})=\Z[1/3]^2$. 
\item $\hat \pi_1(\Omega_\textnormal{SQ})$ is isomorphic to the profinite fundamental group of its
maximal equicontinuous factor, the 2-dimensional 3-adic solenoid. Since the squiral tiling has coincidence
rank 2, all covers of the squiral tiling also have coincidence rank 2. 
\end{enumerate}

Interestingly, the quotient of $\Omega_\textnormal{SQ}$ by the $0 \leftrightarrow 1$ involution has
a more complicated profinite fundamental group, with $\pi_1(\Gamma)$ having 5 generators and 7 relations. 
This quotient has many non-abelian covers, but those covers do not factor through $\Omega_\textnormal{SQ}$.
\end{example}

\begin{example}[The chair tiling]
The chair tiling is MLD to a $2 \times 2$ block substitution. Tiles are unit squares decorated by 
arrows pointing
northeast, northwest, southeast, or southwest. The last three should be viewed as rotations of the northeast arrow.
The northeast arrow substitutes as  
\[
\begin{tikzpicture}
\draw(0,0)--(.50,0)--(.50,.50)--(0,.50)--(0,0);
\draw[->](.1250,.1250)--(.3750,.3750);
\draw[|->](.750,.250)--(1.250,.250);
\draw(1.50,-.250)--(2.50,-.250)--(2.50,.750)--(1.50,.750)--(1.50,-.250);
\draw(2.0,-.250)--(2.0,.750);
\draw(1.50,.250)--(2.50,.250);
\draw[->](1.6250,-.1250)--(1.8750,.1250);
\draw[->](1.6250,.6250)--(1.8750,.3750);
\draw[->](2.1250,.3750)--(2.3750,.6250);
\draw[->](2.3750,-.1250)--(2.1250,.1250);
\end{tikzpicture}
\]
and the substitutions of the other tiles are obtained by rotating the northeast supertile. 
The AP complex of its dual has 4 vertices, \(8+8\) edges, and 19 2-cells. The spanning tree
of the 1-skeleton consists of 3 edges, so $\pi_1(\Gamma)$ has 16 generators and \(19+3\) relations. Fortunately,
this presentation can be reduced to one with two generators, $a$ and $b$, corresponding to 
horizontal and vertical edges. These do not completely commute, but are subject to the relations
\(ba^{-2}bab^{-2}a\), \(aba^{-2}b^{-2}ab\), and \(ba^{-1}b^{-1}ab^{-1}a^{-1}ba\). That is, 
\[
\pi_1(\Gamma)=\left\langle a,b\left|ba^{-2}bab^{-2}a,aba^{-2}b^{-2}ab,ba^{-1}b^{-1}ab^{-1}a^{-1}ba\right.\right\rangle.
\]
The substitution map 
sends each generator to its square. With this finitely presented $\pi_1(\Gamma)$, 
which is almost abelian, we find that the smallest non-abelian groups which 
occur as deck groups are two groups of order 27, each with 48 distinct covers 
of the chair tiling (up to conjugation). The next smallest non-abelian deck 
groups have order 81. In fact, computing up to finite groups of order \(500\), it appears that 
the non-abelian deck groups have the form 
\(\mathbb{Z}/p\mathbb{Z}\times G\), where \(p\) could be \(1\) or a prime \(\neq 3\), 
and \(G\), of order a power of 3, is formed from (nontrivial) semidirect products of 
cyclic groups. 

We now turn to the simplest abelian covers, which have index 3. These are classified by 
$\check H^1(\Omega_\textnormal{ch}, \Z/3\Z) = (\Z/3\Z)^2$. Excluding $(0,0)$, which gives a disconnected 
cover, there are eight possibilities. The covers corresponding to $(0,1)$, $(1,0)$, $(-1,0)$ and $(0,-1)$ 
and all related by rotation, as are the covers corresponding to $(1,1)$, $(-1,1)$, $(-1,-1)$ and $(1,-1)$. 
This leaves two covers to consider, $\tilde \Omega_v$ corresponding to $(0,1)$ and $\tilde \Omega_d$
corresponding to $(1,1)$, where ``\(v\)'' and ``\(d\)'' stand for ``vertical'' and ``diagonal''. 

In both cases, our tiles are arrows with subscripts in $\Z/3\Z$ and the following rules are simple. 
For $\tilde \Omega_v$, each tile has the same subscript as the tile to its left and a subscript one more
than the tile below it. For $\tilde \Omega_d$, each tile has a subscript one more than the tile to its left
and one more than the tile below it. The return vectors of 
tiles of a given subscript form a sublattice of $\Z^2$ of 
index 3, indicating that we have a triple cover of the maximal equicontinuous factor. Thus both 
$\tilde \Omega_v$ and $\tilde \Omega_d$ have pure point spectrum. 

Where $\tilde \Omega_v$ and $\tilde \Omega_d$ differ most from $\Omega_\textnormal{ch}$ is in their
cohomology. Both covers have 3-torsion in $\check H^2$ and this torsion is easily understood. In 
$\Omega_\textnormal{ch}$ there is a 2-cochain $\alpha$ that evaluates to 1 on one of the three tiles
in each L-shaped chair (say, the one that touches the other two) and 0 on the other tiles. $3\alpha$
is cohomologous to a cochain that evaluates to 1 on every tile. 
On $\tilde \Omega_v$ and $\tilde \Omega_d$, let $\tilde \alpha$ be the pullback of $\alpha$, counting
chairs regardless of subscript.

On each cover, let $\tilde \beta$ be a cochain that evaluates to 1
on each tile with a subscript 0. This is cohomologous to a cochain that counts tiles with subscript 1, 
or a cochain that counts tiles with subscript 2, so 3 $\tilde \beta$ is cohomologous to 
a cochain that evaluates to 1 on every tile. The cochains $\tilde \alpha$ and $\tilde \beta$ are 
not cohomologous, yet the classes of $3\tilde \alpha$ and $3\tilde \beta$ are the same, so the class
of $\tilde \alpha - \tilde \beta$ is a torsion element. 

 Something similar happens in several of the nine variants of the chair tiling from \cite{MR2836783} and 
 \cite{BugarinE.P.2014QCoC}, where they show that exactly the ``\(X,+\)'', ``\(/,+\)'', ``\(X,-\)'', 
 ``\(/,-\)'', and ``\(X,0\)'' variants have \(\frac{1}{3}\mathbb{Z}[1/4]\) in their respective second 
 cohomology groups, meaning that there is an element of $\check H^2$ that plays the role of $\alpha$. 
 For each triple cover of one of these spaces, we can define $\tilde \beta$ and 
 $\tilde \alpha - \tilde \beta$ is a torsion element of $\check H^2$ of that cover. 
 
The cohomology groups of the two covers of $\Omega_\textnormal{ch}$ are
\begin{align*}
\check{H}^2(\tilde{\Omega}_v)&=\mathbb{Z}[1/4]\oplus\mathbb{Z}[1/2]_2^2\oplus\mathbb{Z}_1^2\oplus\mathbb{Z}_{-1}^2\oplus(\mathbb{Z}/3\mathbb{Z})_4\\
\check{H}^2(\tilde{\Omega}_d)&=\mathbb{Z}[1/4]\oplus\mathbb{Z}[1/2]_2^3\oplus\mathbb{Z}[1/2]_{-2}\oplus\mathbb{Z}_1\oplus\mathbb{Z}_{-1}\oplus(\mathbb{Z}/3\mathbb{Z})_4
\end{align*}
where the subscripts indicate the less obvious eigenvalues.
Besides the appearance of 3-torsion in both cases, note the many factors of $\Z[1/2]$ in 
$\check H^2(\tilde \Omega_d)$. This is because the factors of $\Z[1/2]$ in $\check H^2(\Omega_\textnormal{ch})$
refer to features that run along diagonal lines and scale by 
2 with substitution. In $\tilde \Omega_v$, a sum of the two diagonal generators becomes divisible by 3,
giving us $\frac13 \Z[1/2] \oplus \Z[1/2]$. In $\tilde \Omega_d$, the subscripts on the 
features with slope $-1$ are constant, so we get one copy of $\Z[1/2]$ for features with slope 1 and
three copies of $\Z[1/2]$ for features with slope $-1$, one for each possible subscript. 

\end{example}

\begin{example}[Products of period-doubling and Thue--Morse]
An easy way to get a 2-dimensional block substitution is as the product of two 
1-dimensional substitutions. $\hat \pi_1$ of a product space is the product of $\hat \pi_1$ of 
its factors. However, not all covers of product spaces are product spaces! 

To see this, consider double covers of $\Omega_\textnormal{PD} \times \Omega_\textnormal{PD}$. 
Since 
\[ \check H^1 (\Omega_\textnormal{PD} \times \Omega_\textnormal{PD}, \Z/2\Z) 
= \check H^1( \Omega_\textnormal{PD}, \Z/2\Z)^2 = (\Z/2\Z)^2, \]
there are three non-trivial double covers of $\Omega_\textnormal{PD} \times \Omega_\textnormal{PD}$
corresponding to the elements $(1,0)$, $(0,1)$ and $(1,1)$. The first is 
$\Omega_\textnormal{TM} \times \Omega_\textnormal{PD}$ and the second is $\Omega_\textnormal{PD} \times \Omega_\textnormal{TM}$,
but the third is not a product. It is instead the 2-dimensional Thue--Morse tiling
space $\Omega_\textnormal{TM2d}$, given by the block substitution
\[
\begin{tikzpicture}
\draw(0,0)--(.50,0)--(.50,.50)--(0,.50)--(0,0);
\draw(.25,.25)node{\(0\)};
\draw[|->](.750,.250)--(1.250,.250);
\draw(1.50,-.250)--(2.50,-.250)--(2.50,.750)--(1.50,.750)--(1.50,-.250);
\draw(2.0,-.250)--(2.0,.750);
\draw(1.50,.250)--(2.50,.250);
\draw(1.75,0)node{\(0\)};
\draw(1.75,.5)node{\(1\)};
\draw(2.25,0)node{\(1\)};
\draw(2.25,.5)node{\(0\)};
\draw(3.5,0)--(4.00,0)--(4.00,.50)--(3.5,.50)--(3.5,0);
\draw(3.75,.25)node{\(1\)};
\draw[|->](4.250,.250)--(4.750,.250);
\draw(5.00,-.250)--(6.00,-.250)--(6.00,.750)--(5.00,.750)--(5.00,-.250);
\draw(5.5,-.250)--(5.5,.750);
\draw(5.00,.250)--(6.00,.250);
\draw(5.25,0)node{\(1\)};
\draw(5.25,.5)node{\(0\)};
\draw(5.75,0)node{\(0\)};
\draw(5.75,.5)node{\(1\)};
\end{tikzpicture}
\]
These spaces can also be understood as quotients of $\Omega_\textnormal{TM} \times \Omega_\textnormal{TM}$, 
a space with two commuting involutions; one switches 0's and 1's in the first factor, while the other
switches 0's and 1's in the second factor. There is also a diagonal involution that switches 0's and 1's 
in both factors. The quotients of $\Omega_\textnormal{TM} \times \Omega_\textnormal{TM}$ by these three
involutions are $\Omega_\textnormal{PD} \times \Omega_\textnormal{TM}$, $\Omega_\textnormal{TM} \times \Omega_\textnormal{PD}$ and $\Omega_\textnormal{TM2d}$. 

We next consider covers of $\Omega_\textnormal{TM2d}$. In the dual tiling,
$\Gamma$ has two vertices, 8 edges,
and 8 tiles. A spanning tree
consists of a single edge, connecting the two vertices, so the fundamental
group of $\Gamma$ has 8 generators and $8+1$ relations. This can be simplified
to a group with 4 generators (2 horizontal, 2 vertical) and
5 relations. We then consider the endomorphism of this group induced by
substitution. These answers are simple enough to determine the number of covers with
small deck groups. For instance, there are two 2-fold covers ($\Omega_\textnormal{TM} \times \Omega_\textnormal{TM}$
and a quotient of $\Omega' \times \Omega'$ by a $\Z/4\Z$ diagonal action, 
where $\Omega'$ is the unique double cover of $\Omega_\textnormal{TM}$), and four with
deck group $A_4$.

\end{example}

\indent For the Penrose tiling, we will make use of the following proposition.

\begin{prop}
\label{prop:induced-isomorphism}
If \(\phi:G\rightarrow G\) is surjective, then \(\hat{\phi}:\hat{G}\rightarrow\hat{G}\) is an isomorphism.\\
\indent In particular, if \(X\) is a finite CW-complex and \(X^1\) its \(1\)-skeleton, and \(\phi:X\rightarrow X\)
is cellular with \(\phi_\ast^1:\pi_1(X^1)\rightarrow\pi_1(X^1)\) surjective, then
\(\hat{\phi}_\ast:\hat{\pi}_1(X)\rightarrow\hat{\pi}_1(X)\) is an isomorphism.
\begin{proof}
  \(\hat{\phi}\) is surjective, since given any finite quotient \(G/N\) of the codomain, the corresponding quotient
  in the domain is \(G/\phi^{-1}(N)\) by definition. Then, since \(\hat{G}\) is profinite, it is residually finite,
  and is Hopfian, therefore \(\hat{\phi}\) is also injective.\\
  \indent For the second part, \(\phi\) cellular implies that the normal subgroup \(N\) generated by attaching the
  \(2\)-cells is preserved under \(\phi\), i.e. \(\phi(N)\leq N\). \(\phi_\ast^1\) surjective then descends to a surjective
  \(\phi_\ast:\pi_1(X)\rightarrow\pi_1(X)\), where \(\pi_1(X)=\pi_1(X^1)/N\).
  By the first part, the induced map on the profinite completion is an isomorphism.
\end{proof}
\end{prop}

\begin{example}[Penrose]
Using the Robinson triangle decomposition of the Penrose tiling, one can show that \(\pi_1(\Gamma)\) has \(5\) generators
\[
\pi_1(\Gamma)=\left\langle a,b,c,d,e\left|\begin{array}{c}
a^{-1}b^{-1}cadbd^{-1}c^{-1},
ceae^{-1}bc^{-1}a^{-1}b^{-1}\\
adb^{-1}ea^{-1}be^{-1}d^{-1},
dc^{-1}be^{-1}d^{-1}ecb^{-1}\\
adc^{-1}a^{-1}d^{-1}ece^{-1},
e^{-1}bdc^{-1}ecb^{-1}d^{-1}\\
dbd^{-1}a^{-1}b^{-1}cd^{-1}adc^{-1},
e^{-1}c^{-1}be^{-1}adecd^{-1}a^{-1}eb^{-1}
\end{array}\right.\right\rangle,
\]
that \(\pi_1(\Gamma^1)=\mathbb{F}_{37}\), and that \(\sigma_\ast^1:\pi_1(\Gamma^1)\rightarrow\pi_1(\Gamma^1)\) is an isomorphism.
By Proposition \ref{prop:induced-isomorphism}, \(\pi_1^\textnormal{et}(\Omega)=\hat{\pi}_1(\Gamma)\). In other words,
all finite, regular covers of the Penrose tiling space arise as finite-index subgroups of \(\pi_1(\Gamma)\).\\
\indent Note that \(\check{H}^1(\Omega)=\mathbb{Z}^5\), which is of rank \(5\), agreeing with our computation here.
\end{example}

\noindent A similar calculation holds for the Ammann chair tiling. This process, unfortunately, does not apply to the
octagonal Ammann--Beenker tiling (using squares and \(45\)-degree rhombi as prototiles), because \(\sigma_\ast^1\) on
\(\pi_1(\Gamma^1)=\mathbb{F}_8\langle f_i\rangle\) is given by $f_i \rightarrow f_i\, f_{i+5}^{-1}\, f_{i+4}^{-1}$,
and is not an isomorphism. One can still use our machinery to compute the finite parts of \(\pi_1^\textnormal{et}(\Omega)\),
but it is unclear what its complete inverse limit structure is.

\appendix

\section{Connection to gauge theory}\label{sec:gauge}

The representation variety $\textnormal{Hom}(X,G)/G$ has a natural interpretation
in terms of gauge theory, as noticed already in \cite{Sad07}. 
Erdin \cite{erdin2010patternequivariantrepresentationvariety} has exploited
this idea to show that $\textnormal{Hom}(X,G)/G$ is a topological invariant.

Here, we use a slightly different, but equivalent approach to elucidate
the connection to gauge theory. Classically, in gauge theory one considers
a connected, smooth manifold $M$, and over it a priciple $G$-bundle
$\pi:P(M,G) \rightarrow M$, with $G$ a connected Lie group 
(for background on gauge theory, see e.g.~\cite{KN63}). We note
that the fibre $\pi^{-1}(x)$ over $x$ is diffeomorphic to $G$, but it
is not a group---there is no preferred point playing the role of the
identity. However, there is a smooth $G$-action on $P$ from the right,
which preserves the fibres, such that for any $p\in P$ the map
$g \rightarrow pg$ is a bijection between $G$ and the fibre containing $p$.
We will need only trivial bundles, meaning that there exists a global
trivialisation of $P$, that is, a diffeomorphism
$P \rightarrow M \times G$, whose pullback defines a coordinate
system on $P$. Such a coordinate system is called a \emph{gauge}.
In general, one only has local trivialisations over charts of $M$,
which one needs to patch together. Equivalent to the choice of a
trivialisation is a choice of a section $s:M \rightarrow P$, such
that $\pi(s(x))=x$ for all $x \in M$. With the help of the right
action of $G$, such a section can then be extended to a trivialisation,
where the points in the section are mapped to the identity of $G$.
Any smooth function $\alpha: M \rightarrow G$ induces a change of
trivialisation $(x,g) \mapsto (x,\alpha(x)g)$, respectively a
change of section $s(x) \mapsto \alpha(x)s(x)$. As this is a
change of coordinate system (gauge), it is called a
\emph{gauge transformation}.

In each tangent space $T_p(P)$ there is a preferred ``vertical''
subspace $V$, consisting of all vectors tangent to the fibre. In
contrast, there is no natural ``horizontal'' complement $H$ to $V$.
This requires an additional structure, a connection
$\mathcal{C}$. A connection defines a smoothly varying field of
horizontal subspaces in the tangent bundle, such that
$H(pg) = g_*(H(p))$, where $g_*$ is the induced action of the
right action of $G$ on the tangent space $T_p(P)$. If $\gamma(t)$
is a path in $M$ from $x$ to $y$, and $p_x$ any point in the fibre
of $x$, with the help of the connection the path $\gamma$ can be
lifted to a unique path $\tilde{\gamma}(t)$ with only horizontal
tangent vectors, and which starts in $p_x$ and ends at some point
$p_y$ in the fibre of $y$. Since, for any $g\in G$, the horizontal
path starting in $p_xg$ ends in $p_yg$, at a fixed gauge the
horizontal lift induces a right translation by some group element
$g$ of the fibre at $y$ relative to the fibre at $x$. This group
element $g$ will depend on the path $\gamma$. We call this the
\emph{horizontal transport} along $\gamma$ of the fibre at $x$ to
the fibre at $y$. In particular, if $\gamma$ is a loop starting
and ending at $x$, the connection associates with it a group
element $g$, such that horizontal transport sends a point $p$
above $x$ to $pg$. The group elements from different loops at
$x$ form a subgroup of $G$, the \emph{holonomy group} of the
connection. Relative to a fixed gauge, the holonomy groups at
different points are conjugate subgroups of $G$. A gauge
transformation maps a holonomy group to one of its conjugates.

A connection is called flat if every contractible loop at $x$
induces the trivial transformation on the fibre at $x$. This
means that loops at $x$ from a single homotopy class, lifted
to the same starting point $p$ above $x$, define the same
transformation on the fibre at $x$, and up to conjugation in
$G$ (to account for the choice of $p$) we have a well-defined
homomorphism $\pi_1(X,x) \rightarrow G$. Conversely, such a
homomorphism uniquely determines a flat connection, up to
gauge transformations. So, we have the well-known result that
\[
  \textnormal{Hom}(\pi_1(M),G)/G \simeq \{ \textrm{flat connections on } P \}
  / \{ \textrm{gauge transformations} \}.
\]

We will now show that an analogous result holds also for tiling
spaces. In the classical theory above, a connection is usually
defined via a connection 1-form. Here, instead of trying to
bring differential forms to tiling spaces (which is possible),
we will use an integral, rather than a differential form of gauge
theory. For this purpose we use the fact that tilings, as well
as their approximant spaces, have a natural cell complex structure.
In physics, this integral from of gauge theory is known as lattice 
gauge theory, where the underlying cell complex is assumed to be
derived from a lattice.

To start with, we assume for the moment that $X$ is a finite cell
complex. We recall that a connection is a device to define horizontal
transport along a path $\gamma$ in base space. Since we are only
interested in topological questions, we can confine the path
to be contained in the $1$-skeleton of $X$, starting and ending
at a $0$-cell (vertex). We also note, that in such a discrete
setting, there is no need for the group $G$ to be differentiable.
We can also work with a discrete, or even finite group $G$.
At a fixed gauge, a connection on the $G$-bundle $X \times G$
is then given by a function $\beta$, which assigns to each
$1$-cell $c$ a group element $\beta(c)$. If $\gamma$ is a path
consisting of the concatenation of a sequence of $k$ $1$-cells
$\{ c_i : i=1,\ldots,k \}_{i\ \in I}$ (in a fixed order), we
extend the definition of $\beta$ by setting $\beta(\gamma) =
\Pi_{i=1}^k \beta(c_i)$. $\beta(\gamma)$ is then the group
element describing the horizontal transport along $\gamma$.
A flat connection is one for which the path-ordered product
$\Pi_{c \in \partial C} \beta(c_i)$ is the identity for every
$2$-cell $C$ in the complex (for a non-flat connection, the curvature
sits on $2$-cells $C$, and is given by the (conjugacy class of)
$\Pi_{c \in \partial C} \beta(c_i)$). If the connection is flat,
it defines a homomorphism $\pi_1(X,x) \rightarrow G$, which
is well-defined up to conjugation in $G$, to make up for the
choice of a starting point of the lift of a loop. A gauge transformation
now is given by a function $\alpha$ assigning a group element
to every $0$-cell. Under such a gauge transformation, if
$\gamma$ is a path starting at $x$ and ending at $y$,
$\beta(\gamma)$ is sent to $\alpha(x)\beta(\gamma)\alpha^{-1}(y)$.
As in the differentiable manifold case, we now have an isomorphism
\[
  \textnormal{Hom}(\pi_1(X),G)/G \simeq \{ \textrm{flat connections on\ } X \times G \}
  / \{ \textrm{gauge transformations} \}.
\]

So far, this equivalence holds for $X$ a finite cell complex, but
obviously this can be extended on either side to inverse limit spaces,
so that the above formula also holds for $X$ a general tiling space, and
$\pi_1(X)$ replaced by the profinite fundamental group  $\hat{\pi}_1(X)$.

\section*{Acknowledgments}

The early stages of this work have profited enormously from discussions with
Eike Lau, who pointed out the relevance of profinite fundamental groups and
finite regular covers in the context of inverse limit spaces.

This work was supported by Deutsche Forschungsgemeinschaft (DFG, German Research 
Foundation), via TRR 358/1 2023--491392403 (FG), and NSF DMS-1937215 (JL,LS).

\printbibliography

\end{document}